\documentclass[10pt,oneside,leqno]{amsart}
\usepackage{amsxtra}
\usepackage{amsopn}
\usepackage{color}
\usepackage{amsmath,amsthm,amssymb}
\usepackage{amscd}
\usepackage{amsfonts}
\usepackage{latexsym}
\usepackage{verbatim}
\usepackage{graphicx}
\usepackage{graphics}
\usepackage{multirow}
\usepackage{lscape}

\usepackage{tikz}
\usetikzlibrary{arrows,positioning,patterns}

\theoremstyle{plain}
\newtheorem{theorem}{Theorem}[section]

\newtheorem{proposition}[theorem]{Proposition}
\newtheorem{corollary}[theorem]{Corollary}
\newtheorem{remark}[theorem]{Remark}
\newtheorem{notation}[theorem]{Notation}

\newtheorem{remark-question}[section]{Remark-Question}

\newcommand\nre{{\nabla^{\varepsilon,\rho}}}
\newcommand\nrex{{\nabla^{\varepsilon,\rho}_X}}
\newcommand\cre{{\Omega^{\varepsilon,\rho}}}
\newcommand\sre{{\sigma^{\varepsilon,\rho}}}
\newcommand\re{{\varepsilon,\rho}}

\newcommand\C{{\mathbb C}}

\newcommand\R{{\mathbb R}}

\newcommand\frg{{\mathfrak g}}
\newcommand\frh{{\mathfrak h}}

\newcommand\Real{{\mathfrak R}{\frak e}\,} 
\newcommand\Imag{{\mathfrak I}{\frak m}\,}

\newcommand{\slC}{\mathfrak{sl}(2,\mathbb C)}

\definecolor{fondo}{rgb}{0.93,0.93,0.93}

\begin{document}
\title[]{Invariant solutions to the Strominger system and the heterotic equations of motion 
}
\subjclass[2000]{
}

\author{Antonio Otal}
\address[A. Otal and R. Villacampa]{Centro Universitario de la Defensa\,-\,I.U.M.A., Academia General
Mili\-tar, Crta. de Huesca s/n. 50090 Zaragoza, Spain}
\email{aotal@unizar.es}
\email{raquelvg@unizar.es}

\author{Luis Ugarte}
\address[L. Ugarte]{Departamento de Matem\'aticas\,-\,I.U.M.A.\\
Universidad de Zaragoza\\
Campus Plaza San Francisco\\
50009 Zaragoza, Spain}
\email{ugarte@unizar.es}

\author{Raquel Villacampa}


\let\thefootnote\relax

\maketitle

\begin{abstract}
We construct many new invariant solutions to the Strominger system
with respect to a 2-parameter family of metric connections $\nre$ in the anomaly cancellation equation. The ansatz $\nre$ is a natural
extension of the canonical 1-parameter family of Hermitian connections found by Gauduchon, as one recovers
the Chern connection $\nabla^{c}$ for $(\re)=(0,\frac12)$, and the Bismut connection $\nabla^{+}$ for $(\re)=(\frac12,0)$.
In particular, explicit invariant solutions to the
Strominger system with respect to the Chern connection, with non-flat instanton and positive~$\alpha'$ are obtained.
Furthermore, we give invariant solutions to the
heterotic equations of motion with respect to the Bismut connection. 
Our solutions live on three different compact non-K\"ahler homogeneous spaces,
obtained as the quotient by a lattice of maximal rank of a nilpotent Lie group,
the semisimple group SL(2,$\mathbb{C}$) and a solvable Lie group.
To our knowledge, these are the only known invariant solutions
to the heterotic equations of motion,
and we conjecture that there is no other such homogeneous space admitting
an invariant solution to the heterotic equations of motion
with respect to a connection in the ansatz $\nre$.

\end{abstract}




\section{Introduction}

\noindent
The goal of this paper is to provide new invariant solutions to the Strominger system and the heterotic equations of motion
in dimension six.
Strominger and Hull investigated, independently in~\cite{Str} and~\cite{Hull},
the heterotic superstring
background with non-zero torsion, which led to a complicated system of partial differential equations.
This system specifies, in six dimensions, the geometric inner space $X$ to be a compact
complex
conformally balanced manifold with holomorphically trivial canonical bundle, equipped with an instanton
compatible with the Green-Schwarz anomaly cancellation condition. The latter
condition, also known as the Bianchi identity, reads as the equation of 4-forms
\begin{equation}\label{anomaly-canc}
dT=\frac{\alpha'}{4}\left({\rm tr}\, \Omega\wedge\Omega -{\rm tr}\, \Omega^A\wedge \Omega^A
\right),
\end{equation}
for some real non-zero constant $\alpha'$. The 3-form $T$ is the torsion of the Bismut connection
of the Hermitian metric, which is given by $T=JdF$, where $J$ is the complex structure of $X$ and
$F$ denotes the fundamental 2-form of the conformally balanced Hermitian metric on $X$.
The form $\Omega$ is the curvature form of some metric connection
$\nabla$ and $\Omega^A$ is the curvature form of the instanton $A$, i.e. $\Omega^{A}$ satisfies the Hermitian-Yang-Mills equation.
For more results on heterotic flux compactifications and on the general theory of the Strominger system, see~\cite{BBDGE}
and the recent paper \cite{MGarcia}, as well as the references therein.

Li and Yau obtained in \cite{Li-Yau} the first non-K\"ahler solutions to the Strominger system on a K\"ahler Calabi-Yau manifold.
This was further extended in~\cite{A-Gar}, where a perturbative method for certain K\"ahler Calabi-Yau threefolds with stable holomorphic vector bundles
is used to prove existence of solutions to the Strominger system
and the equations of motion.
On the other hand,
based on a construction in
\cite{GP}, Fu and Yau first proved the existence of solutions to the
Strominger system on non-K\"ahler Calabi-Yau inner spaces given as a
$\mathbb{T}^2$-bundle over a $K3$ surface~\cite{Fu-Yau}.

There are several connections $\nabla$ proposed for the anomaly cancellation equation \eqref{anomaly-canc},
as the Chern connection $\nabla^c$, the (Strominger-)Bismut connection $\nabla^+$,
the Levi-Civita connection $\nabla^{LC}$ or the connection $\nabla^-$
(see for instance~\cite{Str,GP,Fu-Yau,CCDLMZ,FIUV09} and the references therein).
Recently, Fei and Yau
propose in \cite{FeiYau} to find solutions to the Strominger system with respect to any connection in the canonical 1-parameter family
of Hermitian connections $\nabla^{\textsf{t}}$ on the tangent bundle found by Gauduchon in \cite{Gau}.
The family $\nabla^{\textsf{t}}$ includes the Chern connection ($\textsf{t}=1$)
and the Bismut connection ($\textsf{t}=-1$).
In order to present a more complete and
unified study of the solutions to the Strominger system, in this paper
we will consider a 2-parameter family of metric linear connections $\nre$ that incorporates the family $\nabla^{\textsf{t}}$
and also the connections $\nabla^{-}$ and $\nabla^{LC}$.
More concretely, $\nabla^{\textsf{t}}$ corresponds to $\nabla^{\varepsilon,\frac12-\varepsilon}$ (i.e. $\rho=\frac12-\varepsilon$
and $\textsf{t}=1-4\varepsilon$),
$\nabla^c=\nabla^{0,\frac12}$, $\nabla^{\pm}=\nabla^{\pm\frac12,0}$ and $\nabla^{LC}=\nabla^{0,0}$
(see 
Section~\ref{family-connections} for more details).

An important source of explicit solutions to the Strominger system is provided by the invariant balanced Hermitian geometry on
compact quotients of real 6-dimensional Lie groups $G$ endowed with a left-invariant complex structure $J$.
In this setting the dilaton is constant and the main part of the analysis can be
carried out on the Lie algebra $\frg$ of the Lie group $G$.
The first invariant solutions on nilmanifolds, i.e. the Lie group is nilpotent, with positive $\alpha'$ were found
in \cite{FIUV09}
(see also \cite{CCDLMZ} for solutions on the Iwasawa manifold with $\alpha'<0$).
Recall that
the nilpotent Lie algebras underlying a (non-toral, and thus non-K\"ahler) nilmanifold endowed with a balanced Hermitian
metric with respect to an invariant complex structure are $\frh_2,\frh_3,\frh_4,\frh_5,\frh_6$ and $\frh_{19}^{-}$
(see \cite{UV2} for a description
of the Lie algebras).
The nilmanifolds (corresponding to the Lie algebras) $\frh_2,\frh_3,\frh_4$ and~$\frh_5$
provide some invariant solutions to the Strominger system when one sets the connection $\nabla$
in the anomaly cancellation condition \eqref{anomaly-canc} to be the connection $\nabla^{+}$ or $\nabla^{LC}$
\cite[Theorems 5.1, 5.2 and 6.1]{FIUV09},
whereas the nilmanifold $\frh_6$ has invariant solutions only with respect to $\nabla^{+}$ \cite[Theorem 7.1]{FIUV09}.
In the case of $\frh_{19}^{-}$, in \cite[Theorem 8.2]{FIUV09} some solutions with respect to $\nabla^{+}$ and $\nabla^{c}$
are given.

Except for the case of the Chern connection $\nabla^{c}$ on the nilmanifold $\frh_{19}^{-}$, all the solutions
on nilmanifolds found in \cite{FIUV09} were obtained with a non-flat instanton $A$ defined on the tangent bundle.
Later, in \cite[Theorem 4.2]{UV1} solutions with respect to $\nabla^{c}$, with $\alpha'>0$ and non-flat instanton, were given on $\frh_{19}^{-}$.
More recently, it is proved in \cite{UV2} that \emph{any} invariant balanced
metric compatible with an Abelian complex structure $J$ provides a solution to the Strominger system
with $\alpha'>0$ and non-flat instanton,
when we set the connection $\nabla$
in the anomaly cancellation condition \eqref{anomaly-canc} to be the Bismut connection $\nabla^{+}$.
This result has been used in \cite{FIUV2014} as the starting point
to produce conformally compact and complete smooth solutions to the Strominger system with
non-vanishing flux, non-trivial instanton and \emph{non-constant} dilaton using the first Pontrjagin form
of the connection $\nabla^{-}$.

In \cite{FeiYau}, Fei and Yau provide new invariant solutions on complex Lie groups and their quotients.
For
the semisimple case $\slC$, they
find solutions
with positive $\alpha'$ and flat instanton,
with respect to any connection in the canonical 1-parameter
family of Hermitian connections $\nabla^{\textsf{t}}$
whenever $\textsf{t}<0$.
In particular, solutions
with respect to the Bismut connection ($\textsf{t}=-1$) are found,
which answers a question proposed by Andreas and Garc\'{\i}a-Fern\'andez in \cite{AGarcia}.
Fei and Yau also provide solutions on $\slC$ to the Strominger system with respect
to the Hermitian connection $\nabla^{\textsf{t}}$, with $\alpha'>0$ and non-flat instanton,
for any $\textsf{t}$ satisfying $\textsf{t}(\textsf{t}-1)^2+4<0$. The latter condition excludes both the Chern
connection ($\textsf{t}=1$) and the Bismut connection ($\textsf{t}=-1$).

Ivanov proved in \cite{Iv} (see also \cite{FIUV09}) that a solution of the Strominger system satisfies the heterotic equations
of motion if and only if the connection
$\nabla$ in the anomaly cancellation equation \eqref{anomaly-canc} is an instanton,
i.e. $\Omega$ satisfies the Hermitian-Yang-Mills equation.
As far as we know, the only invariant solutions to the heterotic equations
of motion in the literature are the solutions on the nilmanifold $\frh_3$
found in \cite[Theorems 5.1 and 5.2]{FIUV09}.
In this paper, we find other
two compact homogeneous spaces admitting invariant solutions to the heterotic equations
of motion, which are obtained as the quotient by a lattice of maximal rank
of the semisimple group SL(2,$\mathbb{C}$) and of a solvable Lie group.

In the SL(2,$\mathbb{C}$) case, we prove that an invariant solution found in \cite{FeiYau}
actually provides a solution to the heterotic equations of motion
(with respect to $\nabla^{+}$, with positive $\alpha'$ and the instanton being flat). 
In the solvable case, we consider the Lie algebra denoted as $\frg_7$ in the classification list
obtained in \cite[Theorem 2.8]{FOU} of 6-dimensional solvable Lie algebras
underlying the solvmanifolds with holomorphically trivial canonical bundle.
We find invariant solutions to the
heterotic equations of motion with respect to the Bismut connection $\nabla^{+}$, with $\alpha'>0$ and non-flat instanton.
Moreover, invariant solutions to the
Strominger system with respect to the Chern connection $\nabla^{c}$ are also given on this solvmanifold
(notice that the nilmanifold $\frh_3$ has no solutions for $\nabla^{c}$, see~\cite{FIUV09}).

In conclusion, the invariant solutions of the heterotic equations of motion given in this paper
live on three different compact non-K\"ahler homogeneous spaces which are obtained as the quotient by a lattice of maximal rank
of a nilpotent Lie group (the nilmanifold $\frh_3$),
the semisimple group SL(2,$\mathbb{C}$) and a solvable Lie group (the solvmanifold $\frg_7$).
We conjecture that there is no other such homogeneous space
admitting an invariant solution to the heterotic equations of motion
with respect to a connection in the ansatz $\nre$.

The results in this paper are obtained after a careful analysis of the first Pontrjagin form of the 2-parameter family of connections $\nre$,
and as a consequence many new solutions to the Strominger system with non-flat instanton and $\alpha'$
of different signs are given (see Theorems~\ref{solutions-h3}, \ref{solutions-sl2C}, \ref{solutions-g7-u=0},
\ref{solutions-g7-unot0} and Table~1).

The paper is structured as follows. In Section~\ref{family-connections} we introduce the family of
metric connections $\{\nre\}_{\varepsilon,\rho\in\R}$
which extends the canonical 1-parameter family of Hermitian connections $\nabla^{\textsf{t}}$
and also includes the Levi-Civita and the $\nabla^-$ connections.
After recalling the main ingredients in the Strominger system, we indicate how we will proceed
in our searching of invariant solutions on compact quotients of Lie groups by lattices.

In Section~\ref{h3} we construct many invariant solutions with non-flat instanton on the nilmanifold
$\frak h_3$ with respect to the
connections $\nre$ in the anomaly cancellation equation \eqref{anomaly-canc}.
In particular, we recover the solutions previously found in
\cite{FIUV09, UV2}.

In Section~\ref{sl2C} we revisit the $\slC$ case, extending to the family of connections $\nre$ the study of invariant solutions.
Moreover, we provide solutions to the heterotic equations of motion
with respect to the Bismut connection.

Section~\ref{g7} is devoted to the invariant Hermitian geometry of the solvmanifold $\frg_7$.
We construct many new invariant solutions with respect to the
connections $\nre$ in the anomaly cancellation equation \eqref{anomaly-canc},
in particular, solutions for the Chern connection $\nabla^c$ with non-flat instanton and positive $\alpha'$.
Furthermore, some of our solutions satisfy in addition the heterotic equations of motion.

Finally, in Section~\ref{holonomy-cohomology} we determine the holonomy group of the Bismut connection of the
solutions found in the previous sections. A cohomological property that involves the cup product
by the de Rham cohomology class of the 4-form $F^2$, where $F$ is a balanced metric, is also studied.
At the end of Section~\ref{holonomy-cohomology} a table is included (see Table~1), where
the main results of the paper are gathered.
We also include an Appendix with the curvature 2-forms $(\cre)^i_j$ of any connection $\nre$, which are needed
for the proofs of the results of Sections~\ref{h3}, \ref{sl2C}, \ref{g7} and~\ref{holonomy-cohomology}.



\section{A
family of metric connections and the Strominger system}\label{family-connections}

\noindent In this section we introduce a family of metric connections that extends the connections in the canonical 1-parameter family
of Hermitian connections and also includes other connections which are
of interest in the Strominger system.

Let $(M, J, g)$ be a Hermitian manifold.
A Hermitian connection $\nabla$ is a linear connection defined on the tangent bundle $TM$ such that both the metric and the complex
structure are parallel, i.e. $\nabla g=0$ and $\nabla J=0$.
Gauduchon introduced in \cite{Gau} a 1-parameter family
$\{\nabla^{\textsf{t}}\}_{\textsf{t}\in\R}$ of canonical Hermitian connections
which are distinguished by the properties of the torsion tensor.
This family is given by
\begin{equation*}\label{family_G}
g(\nabla^{\textsf{t}}_XY, Z) = g(\nabla^{LC}_X Y, Z)  + \frac{1-\textsf{t}}{4}\,T(X,Y,Z) + \frac{1+\textsf{t}}{4}\,C(X,Y,Z) ,
\end{equation*}
where $F(\cdot, \cdot) = g(\cdot, J\cdot)$ is the associated
fundamental $2$-form, $C(\cdot, \cdot, \cdot)=dF(J\cdot, \cdot, \cdot)$ denotes the torsion of the Chern connection $\nabla^c$,
and $T(\cdot, \cdot, \cdot)=JdF(\cdot, \cdot, \cdot)=-dF(J\cdot,J\cdot,J\cdot)$
stands for
the torsion $3$-form of the Bismut connection $\nabla^+$. Actually, $\nabla^+$ and $\nabla^c$ are recovered
when $\textsf{t}=-1$ and $\textsf{t}=1$, respectively.

The family of Hermitian connections $\{\nabla^{\textsf{t}}\}_{\textsf{t}\in\R}$ has been recently considered in \cite{FeiYau}
to find solutions to the Strominger system, but
there are other connections which have also been proposed, as for instance
the Levi-Civita connection $\nabla^{LC}$ and the $\nabla^-$-connection
(see \cite{Str,GP,Fu-Yau,CCDLMZ,FIUV09} and the references therein).
This leads us to consider the following extension of the canonical 1-parameter family
of Hermitian connections.
For any $(\varepsilon, \rho) \in \mathbb{R}^2$, we consider the connection $\nre$ given by
\begin{equation}\label{family_2par}
g(\nabla^{\varepsilon,\rho}_XY, Z) = g(\nabla^{LC}_X Y, Z) + \varepsilon\,T(X,Y,Z) + \rho\,C(X,Y,Z).
\end{equation}

In particular, the 1-parameter family $\nabla^{\textsf{t}}$ corresponds to $\nabla^{\varepsilon,\frac12-\varepsilon}$, that is, $\rho=\frac12-\varepsilon$ and $\textsf{t}=1-4\varepsilon$. Note also that
$\nabla^c=\nabla^{0,\frac12}$, $\nabla^{\pm}=\nabla^{\pm\frac12,0}$ and $\nabla^{LC}=\nabla^{0,0}$
(see Figure~\ref{conexiones}).

\medskip

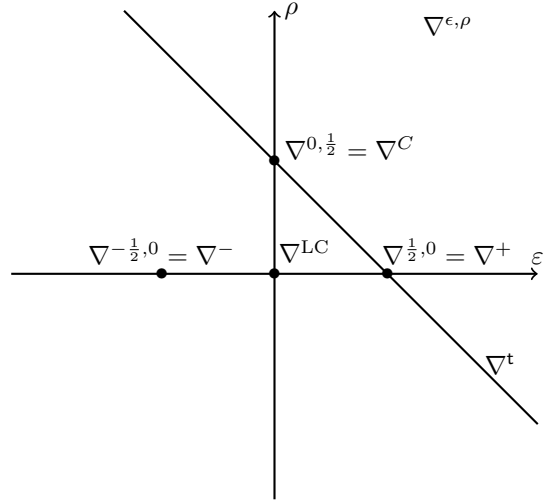
\begin{figure}[h!]
\begin{center}
\begin{tikzpicture}

\draw[-to,thick] (0,-3) -- (0,3.5) node[right] {$\rho$};

\draw[-to,thick] (-3.5,0) -- (3.5,0) node[above] {$\varepsilon$};

\draw[thick] (3.5,-2) -- (-2,3.5);

 \node at (-1.5,0) (nablamenos) {$\bullet$};
  \node at (-1.5,0.3)  {$\nabla^{-\frac12,0}=\nabla^-$};

 \node at (0,0) (LC) {$\bullet$};
 \node at (0.4,0.3)  {$\nabla^\text{LC}$};

 \node at (1.5,0) (nablaB) {$\bullet$};
  \node at (2.3,0.3)  {$\nabla^{\frac12,0}=\nabla^+$};

 \node at (0,1.5) (nablaC) {$\bullet$};
\node at (1,1.7)  {$\nabla^{0,\frac12}=\nabla^C$};

  \node at (2.3,3.3)  {$\nabla^{\epsilon,\rho}$};

  \node at (3,-1.2)  {$\nabla^{\textsf{t}}$};

\end{tikzpicture}
\caption{The family of metric connections $\nabla^{\varepsilon,\rho}$
that extends the canonical 1-parameter family
of Hermitian connections $\nabla^{\textsf{t}}$.}\label{conexiones}
\end{center}
\end{figure}

\medskip

In the following result we obtain the relation between the covariant derivatives of the complex structure~$J$ with
respect to the Levi-Civita connection and the connection $\nre$.

\begin{proposition}
Let $(M,J,g)$ be a Hermitian manifold. For each $(\varepsilon, \rho) \in \mathbb{R}^2$, the connection $\nre$
defined by \eqref{family_2par}
is a linear and metric connection that satisfies
\begin{equation}\label{relacion}
\nre J=-2\Big(\varepsilon+\rho - \frac12\Big)\nabla^{LC}J.
\end{equation}
Therefore, if $(M,J,g)$ is not K\"ahler then, the connection $\nre$ is Hermitian if and only if $\rho=\frac12-\varepsilon$.
\end{proposition}

\begin{proof}
It is straightforward to check that $\nre$ is a linear and metric connection on the tangent bundle $TM$.
The proof of the equality \eqref{relacion} involves a long but standard computation that uses properties of the
Levi-Civita connection, Hermitian metrics and integrable almost-complex structures.
Let us sketch the proof.

Using general properties of the covariant derivative $\nabla J$ of a linear connection $\nabla$ and the definition of $\nre$ we get
\begin{eqnarray*}
g((\nrex J)Y,Z) \!&\!=\!&\! g(\nrex JY ,Z) -g(J (\nrex Y) ,Z) \\[4pt]
\!&\!=\!&\! g(\nrex JY ,Z) + g(\nrex Y, JZ) \\[4pt]
\!&\!=\!&\! g(\nabla^{LC}_X JY ,Z) + g(\nabla^{LC}_X Y , JZ)+\\[3pt] &&(\varepsilon+\rho) (dF(JX, JY, Z) + dF(JX,Y,JZ)).
\end{eqnarray*}

The definition of $F$ and the exterior differential of a 3-form lead to the following relation:
$$
dF(JX, JY, Z) + dF(JX,Y,JZ) = -2(g(\nabla^{LC}_X JY ,Z) + g(\nabla^{LC}_X Y , JZ));
$$
therefore, we can express $g((\nrex J)Y,Z)$ simply as
$$
g((\nrex J)Y,Z) = \Big(1-2(\varepsilon+\rho)
\Big) \left(g(\nabla^{LC}_X JY ,Z) + g(\nabla^{LC}_X Y , JZ)\right).
$$

Now, using properties of the Levi-Civita connection it is possible to transform this expression into
\begin{equation}\label{relation}
g((\nrex J)Y,Z) = \Big(\varepsilon+\rho-\frac12\Big) \left(g(Y, (\nabla^{LC}_X J)Z) - g((\nabla^{LC}_X J)Y, Z)\right).
\end{equation}
Observe that equation~\eqref{relation} holds for any pair $(\varepsilon,\rho)$.
In particular, if $(\varepsilon,\rho) = (0,0)$, $\nre = \nabla^{LC}$ and applying~\eqref{relation} we obtain that
$$g(Y, (\nabla^{LC}_X J)Z)= - g((\nabla^{LC}_X J)Y, Z).$$
Substituting this value in~\eqref{relation} we get
$$g((\nrex J)Y,Z) = -2\Big(\varepsilon+\rho - \frac12\Big)\,g((\nabla^{LC}_X J)Y,Z).$$

Since the Levi-Civita connection is Hermitian if and only if the metric $g$ is K\"ahler,
we deduce finally that $\nabla^{\varepsilon,\rho}J=0$ if and only if $\varepsilon+\rho=\frac12$.
\end{proof}

Let $(M,J,g)$ be a compact Hermitian manifold of complex dimension 3, and let $F$ be the fundamental 2-form.
The Strominger system requires that the compact complex manifold $X=(M,J)$
is endowed with a non-vanishing holomorphic (3,0)-form $\Psi$,
so that $(J,F,\Psi)$ is an SU(3)-structure. Moreover, the following system of
equations must be satisfied \cite{Str}:
\begin{enumerate}
\item[(a)] Gravitino equation: the holonomy of the Bismut connection $\nabla^+$ is contained in
SU(3).
\item[(b)] Dilatino equation: the Lee form $\theta=-\frac{1}{2} Jd^* F$, where
$d^*$ denotes the formal adjoint of $d$ with respect to the
metric $g$, is exact; that is, $\theta=2d\phi$,
$\phi$ being the dilaton function.
\item[(c)] Gaugino equation: there is a Donaldson-Uhlenbeck-Yau
instanton,
that is, a connection $A$ with curvature $\Omega^{A}$ satisfying the Hermitian-Yang-Mills equation:
$\Omega^{A}\wedge F^2=0$, $(\Omega^{A})^{0,2}=(\Omega^{A})^{2,0}=0$.
%
\item[(d)] Anomaly cancellation condition, i.e. equation \eqref{anomaly-canc}:
$$
dT=2\pi^2 \alpha' \Big(p_1(\nabla)-p_1(A)\Big),
$$
for some real non-zero constant $ \alpha'$. Here $p_1$ denotes the 4-form representing
the first Pontrjagin
class of the corresponding connection, i.e. in terms of the curvature $\Omega$
it is given by $p_1= {1\over 8\pi^2} {\rm tr}\, \Omega\wedge\Omega$.
\end{enumerate}

The Strominger system was reformulated by Li and Yau in \cite{Li-Yau}, where they showed that instead of
the equations (a) and (b) one can equivalently consider the equation
\begin{equation}\label{eq-Li-Yau}
d(\parallel\!\!\Psi\!\!\parallel_F  F^2)= 0,
\end{equation}
where $\parallel\!\!\Psi\!\!\parallel_F$ is the norm of the form $\Psi$
measured using the Hermitian metric $F$.
The equation \eqref{eq-Li-Yau} implies the existence of a balanced metric $\tilde{F}$
just by modifying the metric $F$ conformally as $\tilde{F}=\parallel\!\!\Psi\!\!\parallel_F^{1/2}  F$.

We will look for solutions to the Strominger system with respect to the connections $\nre$ in the anomaly cancellation condition,
i.e. $\nabla=\nre$.
As we have seen above, this family includes the canonical Hermitian connections $\nabla^{\textsf{t}}$,
the Levi-Civita connection $\nabla^{LC}$ and the
connection $\nabla^-$.
Moreover, a result due to Ivanov in \cite{Iv} asserts that a solution of the Strominger system satisfies the heterotic equations
of motion if and only if the connection $\nabla$ in the anomaly cancellation condition is an instanton.

In the following sections we provide new solutions to the Strominger system and to the heterotic equations
of motion. We will put special attention to solutions with $\alpha'$
positive and non-flat instanton $A$.
Notice that since we look for solutions which are invariant, the dilaton function $\phi$ will always be constant, that is,
the Lee 1-form $\theta$ vanishes identically. The invariance of the non-vanishing holomorphic (3,0)-form $\Psi$ implies
that the function $\parallel\!\!\Psi\!\!\parallel_F$ is constant, too.
By \eqref{eq-Li-Yau} this condition is equivalent to the closedness of the form $F^{2}$,
i.e. the Hermitian structure $F$ is balanced.

\bigskip

From now on, the manifold $M=G/\Gamma$ will be a compact quotient of a Lie group $G$ by a lattice $\Gamma$,
endowed with an invariant Hermitian structure $(J,g)$,
that is, $(J,g)$ can be defined at the level of the Lie algebra $\frg$ of $G$.
Let $\{e^1,\ldots,e^6\}$ be a basis of $\frg^*$ adapted to the Hermitian structure, and let $\{e_1,\ldots,e_6\}$
be its dual basis for $\frg$; that is to say, the complex structure $J$ and the metric $g$ express in this basis as
$$
J e^1=-e^2, \ J e^3=-e^4, \ J e^5=-e^6, \quad\quad g=e^1\otimes e^1 + \cdots +e^6\otimes e^6.
$$
Hence, the fundamental 2-form $F$ is given by $F=e^{12}+e^{34}+e^{56}$. Here, and from now on,
we will denote the wedge product $e^{i_1}\wedge\ldots\wedge e^{i_k}$ briefly by $e^{i_1\,\ldots\, i_k}$.

Let $c_{ij}^k$ be the structure constants of the Lie algebra $\frg$ with respect to the basis $\{e^1,\ldots,e^6\}$,
that is, the structure equations of $\frg$ are
$$
d\,e^k = \sum_{1\leq i<j \leq 6} c_{ij}^k \, e^{i j},\quad\quad
k=1,\ldots,6.
$$

Given any linear connection $\nabla$, the connection 1-forms
$\sigma^i_j$ with respect to the basis above are
$$
\sigma^i_j(e_k) = g(\nabla_{e_k}e_j,e_i),
$$
i.e. $\nabla_X e_j = \sigma^1_j(X)\, e_1 +\cdots+ \sigma^6_j(X)\, e_6$.
The curvature 2-forms $\Omega^i_j$ of $\nabla$ are then given in
terms of the connection 1-forms $\sigma^i_j$ by
\begin{equation}\label{curvature}
\Omega^i_j = d \sigma^i_j + \sum_{1\leq k \leq 6}
\sigma^i_k\wedge\sigma^k_j.
\end{equation}

Now, we provide explicit expressions for the connection 1-forms $(\sre)^i_j$ of the metric connection $\nabla^{\re}$.
Since $d e^k(e_i,e_j)= -e^k([e_i,e_j])$ and the basis $\{e_1,\ldots,e_6\}$ is orthonormal, the Levi-Civita connection
1-forms $(\sigma^{LC})^i_j$ of the metric $g$ express in terms of the
structure constants $c_{ij}^k$ as
$$
(\sigma^{LC})^i_j(e_k) = -\frac12 \left( g(e_i,[e_j,e_k]) - g(e_k,[e_i,e_j]) +
g(e_j,[e_k,e_i]) \right)=\frac12(c^i_{jk}-c^k_{ij}+c^j_{ki}),
$$
and therefore the connection 1-forms $(\sre)^i_j$ of the connection $\nre$ are given by
\begin{eqnarray}\label{connection-1-forms}
\nonumber (\sre)^i_j(e_k)&=&(\sigma^{LC})^i_j(e_k) + \varepsilon\, T(e_k,e_j,e_i) +\rho\, C(e_k, e_j, e_i) \\ &=&
\frac12(c^i_{jk}-c^k_{ij}+c^j_{ki}) -\varepsilon\, T(e_i, e_j, e_k) - \rho\, C(e_k,e_i,e_j)\\ &=&
\nonumber \frac12(c^i_{jk}-c^k_{ij}+c^j_{ki}) +\varepsilon\,dF(J e_i, J e_j, J e_k) - \rho\, dF(J e_k,e_i,e_j).
\end{eqnarray}

\medskip

The connections $A$ satisfying the gaugino equation, i.e. equation (c) in the Strominger system, that we will consider in this paper
are all defined on the tangent bundle. Moreover, the connection $A$ will be an SU(3)-connection, i.e. $A$ will be compatible with
the SU(3)-structure $(J,F,\Psi)$. One can easily express both the latter compatibility condition and the gaugino condition
in terms of a basis adapted to the SU(3)-structure, i.e. in terms of a basis $\{e^1,\ldots,e^6\}$ satisfying
$$
J e^1=-e^2, \ J e^3=-e^4, \ J e^5=-e^6,
\quad F=e^{12}+e^{34}+e^{56},
\quad \Psi=(e^{1}+i e^{2})\wedge (e^{3}+i e^{4})\wedge (e^{5}+i e^{6}).
$$
Indeed, if $(\sigma^{A})^i_j$ are the connection 1-forms of a linear connection $A$ in the adapted basis
and $(\Omega^A)^i_j$ the curvature 2-forms,
then
\begin{enumerate}
\item[(i)]
the connection $A$ is compatible with the SU(3)-structure, i.e. $AJ=AF=A\Psi=0$,
\emph{if and only if} the connection 1-forms satisfy the following identities:
\begin{equation}\label{A-su3_structure}
\begin{array}{l}
(\sigma^{A})^j_i = -(\sigma^{A})^i_j, \quad (\sigma^{A})^1_2 + (\sigma^{A})^3_4 + (\sigma^{A})^5_6 = 0,
\\[1em]
(\sigma^{A})^1_3 = (\sigma^{A})^2_4, \quad (\sigma^{A})^1_4 = -(\sigma^{A})^2_3, \quad
(\sigma^{A})^1_5 = (\sigma^{A})^2_6, \quad (\sigma^{A})^1_6 = -(\sigma^{A})^2_5,
\\[1em]
(\sigma^{A})^3_5 = (\sigma^{A})^4_6, \quad (\sigma^{A})^3_6 = -(\sigma^{A})^4_5;
\end{array}
\end{equation}
\item[(ii)]
the connection $A$ is an instanton, i.e. $\Omega^{A}\wedge F^2=0$ and $(\Omega^{A})^{0,2}=(\Omega^{A})^{2,0}=0$,
\emph{if and only if}
the curvature 2-forms
satisfy the following conditions:
\begin{equation}\label{su3_instanton}
\begin{array}{l}
(\Omega^{A})^i_j(e_1,e_2) + (\Omega^{A})^i_j(e_3,e_4) + (\Omega^{A})^i_j(e_5,e_6)=0,
\\[1em]
(\Omega^{A})^i_j(Je_k,Je_l) - (\Omega^{A})^i_j(e_k,e_l)=0,
\end{array}
\end{equation}
for any $1 \leq i,j,k,l \leq 6$.
\end{enumerate}

\section{The nilmanifold $\frak h_3$}\label{h3}

\noindent
In this section we construct many invariant solutions on a nilmanifold
when we set the connection $\nabla$
in the anomaly cancellation equation to be a connection in the ansatz $\nre$.
In particular, we recover the solutions previously found in
\cite{FIUV09,UV2}.

We recall that a nilmanifold is a compact quotient of a simply-connected nilpotent Lie group $G$ by a lattice $\Gamma$
of maximal rank. If $\frg$ is the Lie algebra of $G$, then any structure defined on $\frg$ will descend
naturally to an \emph{invariant} structure on the nilmanifold. Here we take $G$ as the product Lie group of $\mathbb{R}$
by the 5-dimensional generalized Heisenberg group.
We will denote by $\frak h_3$ the Lie algebra of $G$.

Let us recall that
$\frak h_3$ has, up to isomorphism, two complex structures $J^\pm$,
but only $J^-$ admits balanced metrics. There is a $(1,0)$-basis $\{\omega^i\}_{i=1}^3$
for which the complex equations of $J^-$ are
\begin{equation}\label{J-h3}
J^- \ \colon \quad d\omega^1 = d\omega^2 = 0,\quad d\omega^3 = \omega^{1\bar1} - \omega^{2\bar2}.
\end{equation}
Moreover, by \cite[Lemma 2.10]{UV2} any balanced $J^-$-Hermitian
structure is isomorphic to one
and only one in the following family:
\begin{equation}\label{2form-h3}
2\, F_t = i\, (\omega^{1\bar1} + \omega^{2\bar2} + t^2\,\omega^{3\bar3}),
\end{equation}
where $t \in \mathbb{R}- \{0\}$.

For each $t \not=0$, we consider the basis of (real) 1-forms $\{e^1,\ldots,e^6\}$ given by
\begin{equation}\label{real-1-basis}
e^1 + i\,e^2 =\omega^1,\quad e^3 + i\,e^4 = \omega^2,\quad  e^5 + i\,e^6=t\,\omega^3.
\end{equation}
This basis is adapted to the balanced structure $(J^-,F_t)$ in the sense that both the complex structure and the
fundamental form express canonically as
\begin{equation}\label{adapted-basis-h3}
J^-(e^1)=-e^2,\ J^-(e^3)=-e^4,\ J^-(e^5)=-e^6,\quad\quad F_t=e^{12}+e^{34}+e^{56}.
\end{equation}
We will consider the holomorphic (3,0)-form $\Psi_t$ given by
\begin{equation}\label{Psi-h3}
\Psi_t=(e^1+i\, e^2)\wedge (e^3+i\, e^4)\wedge (e^5+i\, e^6)= t\,\omega^{123}.
\end{equation}
Hence, for each $t \not=0$, we have an invariant SU(3)-structure $(J^-,F_t,\Psi_t)$ on
a nilmanifold with underlying Lie algebra $\frh_3$.

From \eqref{J-h3} and \eqref{real-1-basis}, the (real) structure equations in terms of the adapted basis
$\{e^j\}_{j=1}^{6}$ are:
\begin{equation}\label{equations-h3-es}
de^j = 0,\quad 1\leq j\leq 5,\quad\  de^6 = -2t (e^{12}-e^{34}).
\end{equation}

\begin{remark}\label{posible-lemma}
{\rm
In~\cite{UV2} the authors found instantons $A_{\lambda}$, $\lambda\in \mathbb{R}$, for each SU(3)-structure $(J^-,F_t,\Psi_t)$.
The connection $A_{\lambda}$ is defined on the tangent bundle by the following connection 1-forms:
$$
(\sigma^{A_{\lambda}})^1_2 = -(\sigma^{A_{\lambda}})^2_1 = -(\sigma^{A_{\lambda}})^3_4 = (\sigma^{A_{\lambda}})^4_3 = \lambda (e^5+ e^6),
$$
and $(\sigma^{A_{\lambda}})^i_j =0$, for any $(i,j)\not=(1,2),(2,1),(3,4),(4,3)$. By \eqref{A-su3_structure} the connection
is compatible with the SU(3)-structure $(J^-,F_t,\Psi_t)$, and one can verify that $A_{\lambda}$ satisfies \eqref{su3_instanton}.
Furthermore,
$$
p_1(A_{\lambda}) = -\frac{2\,t^2\lambda^2}{\pi^2}\,e^{1234},
$$
and the instanton $A_{\lambda}$ is non-flat if and only if $\lambda\not=0$.
}
\end{remark}

\medskip

Let us consider now the family of connections $\nre$ introduced in Section~\ref{family-connections}.
In \cite{FIUV09} it is proved that the Bismut connection $\nabla^{+}$ is an instanton. In the following result we prove that
$\nabla^{+}$ is the only connection in the family $\nre$ satisfying the instanton condition.

\begin{proposition}\label{instan-h3}
The connection $\nre$ is an instanton with respect to an {\rm SU(3)}-structure $(J^-,F_t,\Psi_t)$,
i.e.  $\nre$ satisfies \eqref{su3_instanton},
if and only if $(\varepsilon, \rho) = (\frac12, 0)$, i.e. $\nre~=~\nabla^{+}$.
\end{proposition}

\begin{proof}
Let $(\Omega^{\varepsilon,\rho})_j^i$ be the curvature 2-forms of the connection $\nre$. The form $(\Omega^{\varepsilon,\rho})^1_2$
satisfies (see Appendix~\ref{apendice1} for details):
$$
(\Omega^{\varepsilon,\rho})^1_2(e_1,e_2) + (\Omega^{\varepsilon,\rho})^1_2(e_3,e_4) + (\Omega^{\varepsilon,\rho})^1_2(e_5,e_6)
= -((1-2\varepsilon)^2 + 4 \rho^2) t^2.
$$
Since $t \not=0$, the latter expression vanishes if and only if
$\varepsilon=\frac12$ and $\rho=0$,
that is, if and only if $\nre$ is the Bismut connection $\nabla^+$.
Moreover, looking at the curvature forms given
in Appendix~\ref{apendice1} when $\varepsilon=\frac12$ and $\rho=0$, we see that all the curvature forms
$(\Omega^+)^i_j$ vanish
except for $(\Omega^+)^1_2$ and $(\Omega^+)^3_4$. Since $(\Omega^+)^3_4=-(\Omega^+)^1_2=4 t^2 e^{12}-4 t^2 e^{34}$,
one concludes that the equations~\eqref{su3_instanton} are satisfied for $A=\nabla^{+}$.
\end{proof}

The following result sums up our conclusions on non-trivial solutions to the Strominger system for $\frh_3$ with respect to the
connections $\nre$ in the anomaly cancellation condition (see also Figure~\ref{conexiones-3}).
We put special attention to the solutions with positive $\alpha'$
and to the solutions with respect to the preferred connections (i.e. $\nabla^{LC}$, $\nabla^{\pm}$ and $\nabla^c$)
and the Hermitian connections $\nabla^{\textsf{t}}$
in the anomaly cancellation condition.
We also distinguish the solutions that satisfy the heterotic equations of motion.

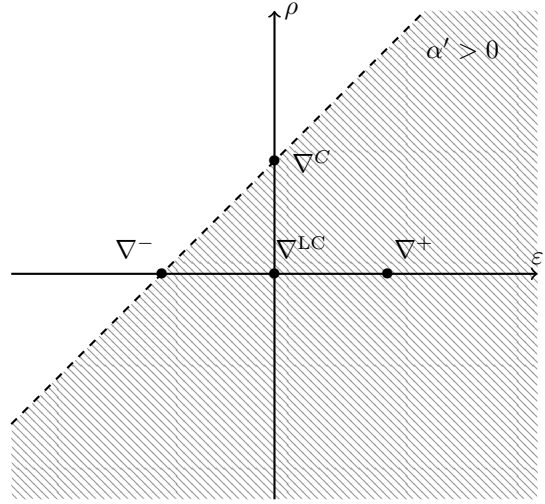
\begin{figure}[h!]

\begin{center}

\begin{tikzpicture}

 \draw[pattern=north west lines, pattern color=white!20!gray,dashed,draw=white] (-3.5,-2) -- (-3.5,-3) -- (3.5,-3) -- (3.5,3.5) -- (2,3.5);	

\draw[-to,thick] (-3.5,0) -- (3.5,0) node[above] {$\varepsilon$};

\draw[-to,thick] (0,-3) -- (0,3.5) node[right] {$\rho$};

\draw[dashed,thick] (-3.5,-2) -- (2,3.5);

 \node at (-1.5,0) (nablamenos) {$\bullet$};
   \node[above left of= nablamenos, node distance=.5cm] {$\nabla^-$};

 \node at (0,0) (LC) {$\bullet$};
   \node[above right of= LC, node distance=.5cm] {$\nabla^\text{LC}$};

 \node at (1.5,0) (nablaB) {$\bullet$};
   \node[above right of= nablaB, node distance=.5cm] {$\nabla^+$};

 \node at (0,1.5) (nablaC) {$\bullet$};
   \node[right of= nablaC, node distance=.5cm] {$\nabla^C$};

   \node at (2.5,3) {$\alpha'>0$};

\end{tikzpicture}
\caption{Solutions on a nilmanifold with underlying Lie algebra $\frh_3$ with respect to the connections $\nre$;
the
Bismut connection $\nabla^+$ satisfies in addition the heterotic equations of motion. The dashed region corresponds to solutions with $\alpha'>0$.
}\label{conexiones-3}
\end{center}

\end{figure}


\begin{theorem}\label{solutions-h3}
On a nilmanifold with underlying Lie algebra $\frh_3$ endowed with an {\rm SU(3)}-structure
given by~\eqref{adapted-basis-h3}-\eqref{Psi-h3}, the Strominger system has invariant solutions
for any connection $\nre$ and with a non-flat instanton $A_{\lambda}$.
More concretely:
\begin{enumerate}
\item[{\rm (i)}]
If $\rho \geq \varepsilon +\frac12$, then there exist solutions to the Strominger system
with respect to the connection $\nre$ in the anomaly cancellation condition, with non-flat instanton
and with $\alpha'<0$.
In particular, there are solutions with respect to
the connection $\nabla^-$ ($\varepsilon=-\frac12,\rho=0$)
and with respect to the Hermitian connection $\nabla^{\textsf{t}}$ for any $\textsf{t} \geq 1$,
which includes the Chern connection $\nabla^c$.
\item[{\rm (ii)}]
If $\rho < \varepsilon +\frac12$, then there exist solutions to the Strominger system
with respect to the connection $\nre$ in the anomaly cancellation condition, with $\alpha'>0$ and with non-flat instanton.
In particular, there are solutions
with respect to the Levi-Civita connection $\nabla^{LC}$ ($\varepsilon=\rho=0$)
and with respect to the Hermitian connection $\nabla^{\textsf{t}}$ for any $\textsf{t} < 1$,
which includes the Bismut connection $\nabla^+$.
\item[{\rm (iii)}]
Furthermore, for the Bismut connection $\nabla^+$, the solutions
satisfy the heterotic equations of motion with $\alpha'>0$ and non-flat instanton.
\end{enumerate}
\end{theorem}

\begin{proof}
Since $T=J^-dF_t=-2t(e^{126}-e^{346})$, we have that $dT=-8t^2 e^{1234}$.

We will use the instantons $A_{\lambda}$ found in~\cite{UV2} (see Remark~\ref{posible-lemma}) to solve the
anomaly cancellation condition with respect to the connections $\nre$.
From the curvature 2-forms of $\nre$ given in Appendix~\ref{apendice1}, it is straightforward to verify that
$$
p_1(\nre) = -\frac{(1+2\varepsilon - 2\rho)(3+4\varepsilon^2-4\rho+4\rho^2)t^4}{\pi^2}\, e^{1234}.
$$
Hence,
$$
2 \pi^2 \alpha' \Big( p_1(\nre) - p_1(A_{\lambda}) \Big)
= 2 t^2 \alpha' \Big( 2\lambda^2 - t^2(1+2\varepsilon - 2\rho)(3+4\varepsilon^2-4\rho+4\rho^2) \Big) e^{1234}.
$$
Comparing this expression with $dT=-8t^2 e^{1234}$, we must prove that there is a non-zero constant $\alpha'$ such that
$$
t^2(1+2\varepsilon - 2\rho)(3+4\varepsilon^2-4\rho+4\rho^2)-2\lambda^2 = \frac{4}{\alpha'}.
$$
Notice that here $\lambda$ is the parameter defining the instanton $A_{\lambda}$, and that $A_{\lambda}$ is non-flat
if and only if $\lambda \not=0$ (see Remark~\ref{posible-lemma}).
Observe also that $3+4\varepsilon^2-4\rho+4\rho^2 = (1-2\rho)^2 + 2(1+2\varepsilon^2)>0$.
Now, taking an instanton $A_{\lambda}$ with sufficiently small non-zero $\lambda$, we conclude that
there exist solutions to the Strominger system with $\alpha'>0$
if and only if
$$
1+2\varepsilon - 2\rho>0.
$$
This proves {\rm (ii)}, and the proof of {\rm (i)} is direct
since $1+2\varepsilon - 2\rho \leq 0$ implies that $\alpha'$ must be negative, after taking any instanton
$A_{\lambda}$ with $\lambda \not=0$.

Finally, if $\nre = \nabla^+$ then by Proposition~\ref{instan-h3} and \cite{Iv} we obtain that the solutions
satisfy the heterotic equations of motion, which proves {\rm (iii)}.
\end{proof}

Notice that this result extends the main results found in \cite{FIUV09} and \cite{UV2} on
a nilmanifold with underlying Lie algebra $\frh_3$ to other
connections $\nre$ in the anomaly cancellation condition.

\section{$\mathfrak{sl}(2,\mathbb C)$ revisited}\label{sl2C}

\noindent
In this section we consider invariant solutions on a compact quotient of the complex Lie group $\text{SL}(2,\C)$.
This manifold has been studied recently in \cite{FeiYau} (see also \cite{AGarcia}), where several solutions to the
Strominger system are obtained with flat as well as with non-flat instanton, with respect to the
family of canonical Hermitian connections $\nabla^{\textsf{t}}$ in the anomaly cancellation condition.
Here we revisit this manifold with two main purposes: to extend the existence of invariant solutions to the
more general family of connections $\nre$, and
to show that this manifold provides solutions to the heterotic equations of motion with respect to the Bismut connection.
The latter was suggested by Andreas and Garc\'{\i}a-Fern\'andez in~\cite{AGarcia}
but, to our knowledge, it has not been proved yet.

The complex(-parallelizable) structure $J$ on ${\rm SL}(2,\mathbb C)$ can be described by means of a left-invariant
basis of $(1,0)$-forms $\{\omega^1,\omega^2,\omega^3\}$
satisfying the equations
\begin{equation}\label{equations-sl2C-bis}
J \ \colon \quad d\omega^1=\omega^{23},\quad d\omega^2=-\omega^{13},\quad d\omega^3=\omega^{12}.
\end{equation}
Since $J$ is complex-parallelizable, it is well known by~\cite{AG} that
any left-invariant Hermitian metric is balanced.
We will consider on the Lie algebra $\mathfrak{sl}(2,\mathbb C)$ the following particular family of
balanced metrics
\begin{equation}\label{2form-sl2C}
2\, F_t = i\, t^2 (\omega^{1\bar1} + \omega^{2\bar2} + \omega^{3\bar3}),
\end{equation}
where $t \in \mathbb{R}- \{0\}$.

For each $t \not=0$, let us consider the basis of (real) 1-forms $\{e^1,\ldots,e^6\}$ given by
\begin{equation}\label{real-1-basis-sl2C-bis}
e^1 + i\,e^2 =t\,\omega^1,\quad e^3 + i\,e^4 =t\,\omega^2,\quad  e^5 + i\,e^6=t\,\omega^3.
\end{equation}
This basis is adapted to the balanced structure $(J,F_t)$ since the complex structure and the
fundamental form express canonically
\begin{equation}\label{adapted-basis-sl2C}
J e^1 =-e^2,\ J e^3 =-e^4,\ J e^5 =-e^6,\quad\quad F_t=e^{12}+e^{34}+e^{56}.
\end{equation}
We consider the holomorphic (3,0)-form $\Psi_t$ given by
\begin{equation}\label{Psi-sl2C}
\Psi_t=(e^1+i\, e^2)\wedge (e^3+i\, e^4)\wedge (e^5+i\, e^6)= t^3\,\omega^{123}.
\end{equation}
Hence, for each $t \not=0$, we have an SU(3)-structure $(J,F_t,\Psi_t)$ on $\mathfrak{sl}(2,\mathbb C)$.

From \eqref{equations-sl2C-bis} and \eqref{real-1-basis-sl2C-bis}, the (real) structure
equations in terms of the adapted basis of 1-forms $\{e^j\}_{j=1}^{6}$ are:
\begin{equation}\label{equations-sl2C-es-bis}
\left\{\begin{array}{ll}
\!\!&\!\! de^1 =\frac{1}{t} \left(e^{35} - e^{46}\right),\\[1em]
\!\!&\!\! de^2=\frac{1}{t} \left(e^{36} + e^{45}\right),\\[1em]
\!\!&\!\! de^3= -\frac{1}{t} \left(e^{15} - e^{26}\right),\\[1em]
\!\!&\!\! de^4= -\frac{1}{t} \left(e^{16} + e^{25}\right),\\[1em]
\!\!&\!\! de^5= \frac{1}{t} \left(e^{13} - e^{24}\right),\\[1em]
\!\!&\!\! de^6 = \frac{1}{t} \left(e^{14} + e^{23}\right).
\end{array}\right.
\end{equation}

Since the complex structure $J$ is complex-parallelizable, the Chern connection $\nabla^{c}$ is flat. In the
following result we prove that there is only one non-flat instanton in the family $\nre$, which is precisely the Bismut connection.

\begin{proposition}\label{instan-sl2C-bis}
The connection $\nre$ is an instanton with respect to an {\rm SU(3)}-structure $(J,F_t,\Psi_t)$
if and only if $(\varepsilon, \rho) = (0,\frac12)$ or $(\frac12, 0)$, i.e. if and
only if $\nre$ is the Chern connection $\nabla^{c}$ or the Bismut connection $\nabla^{+}$. The Chern connection is flat,
but for the Bismut connection one has
$$
p_1(\nabla^+) = -\frac{2}{\pi^2t^4}(e^{1234}+e^{1256}+e^{3456}).
$$
\end{proposition}

\begin{proof}
Let $(\Omega^{\varepsilon,\rho})_j^i$ be the curvature 2-forms of the connection $\nre$,
which are given in Appendix~\ref{apendice2}. The form $(\Omega^{\varepsilon,\rho})^1_2$
satisfies
$$
(\Omega^{\varepsilon,\rho})^1_2(e_1,e_2) + (\Omega^{\varepsilon,\rho})^1_2(e_3,e_4) + (\Omega^{\varepsilon,\rho})^1_2(e_5,e_6)
= -\frac{(1 - 2 \varepsilon -2 \rho)^2}{t^2}.
$$
The latter expression vanishes if and only if
$\rho=\frac12 - \varepsilon$,
that is, if and only if $\nre$ is a Hermitian connection.

Moreover, for $\rho=\frac12 - \varepsilon$, we have that
$$
(\Omega^{\varepsilon,\rho})^1_3(Je_2,Je_4) = (\Omega^{\varepsilon,\rho})^1_3(e_2,e_4)
$$
if and only if
$$
\frac{4 \varepsilon(1 - 4 \varepsilon)}{t^2}=-\frac{4 \varepsilon}{t^2},
$$
which implies that $\varepsilon=0$, and then the connection is $\nabla^c$,
or $\varepsilon=\frac12$, and then the connection is $\nabla^+$.

For the Chern connection all the curvature forms vanish, i.e. $\nabla^c$ is flat.
In the second case, i.e. for the Bismut connection, the non-zero curvature forms
$(\Omega^+)_j^i=(\Omega^{\frac12,0})_j^i$ are:
\begin{equation}\label{Bismut-curv-sl2C}
\begin{array}{rl}
& (\Omega^+)^1_3=(\Omega^+)^2_4=-\frac{2}{t^2}(e^{13}+e^{24}), \\[6pt]
& (\Omega^+)^1_5=(\Omega^+)^2_6=-\frac{2}{t^2}(e^{15}+e^{26}), \\[6pt]
& (\Omega^+)^3_5=(\Omega^+)^4_6=-\frac{2}{t^2}(e^{35}+e^{46}).
\end{array}
\end{equation}
Since the equations~\eqref{su3_instanton} are satisfied for $A=\nabla^{+}$, we conclude that
the Bismut connection is a (non-flat) instanton. Finally, \eqref{Bismut-curv-sl2C} implies that
$$
{\rm tr}\, \Omega^+ \wedge \Omega^+ = -\frac{16}{t^4}(e^{1234}+e^{1256}+e^{3456}).
$$
\end{proof}

\begin{remark}\label{posible-observ}
{\rm
For other balanced metrics on $\mathfrak{sl}(2,\mathbb C)$ more general than
the metrics~$F_t$, we arrived at the conclusion than whenever $\nre$ satisfied the instanton condition,
the metric was isomorphic to $F_t$. This is the reason why we are focusing on the family
of SU(3)-structures $(J,F_t,\Psi_t)$.
}
\end{remark}

Next, we will solve the anomaly cancellation condition with respect to the connections $\nre$.
From the curvature 2-forms of $\nre$ given in Appendix~\ref{apendice2}, it is straightforward to verify that
$$
p_1(\nre) = -\frac{\beta(\re)}{4\pi^2t^4} (e^{1234} + e^{1256} + e^{3456}),$$
where
\begin{equation}\label{beta-bis}
\beta(\re)= 1 + 4\varepsilon  + 4 \varepsilon^2 + 32 \varepsilon^3 - 12\rho- 24\varepsilon\rho -
 32\varepsilon^2\rho + 36\rho^2 + 32\varepsilon\rho^2 - 32\rho^3.
\end{equation}

\medskip

In the following result we obtain many solutions to the Strominger system, including solutions to the heterotic equations of motion.
Special attention is given to the solutions with positive $\alpha'$ and to the solutions with respect to the preferred
and Hermitian connections
in the anomaly cancellation condition.
We distinguish the case
when the instanton is flat and the case when the instanton is non-flat.

\begin{theorem}\label{solutions-sl2C}
On a compact quotient of ${\rm SL}(2,\mathbb C)$ endowed with an {\rm SU(3)}-structure
given by~\eqref{adapted-basis-sl2C}-\eqref{Psi-sl2C}, the Strominger system has invariant solutions
for any connection~$\nre$.
More concretely:
\begin{enumerate}
%
%
\item[{\rm (i)}] For any $(\re)\in \mathbb{R}^2$ such that $\beta(\re) \neq 0$, there exist solutions to the Strominger system
with respect to the connection $\nre$ in the anomaly cancellation condition, with flat instanton and ${\rm sign}\, (\alpha') = {\rm sign} \left( \beta(\re) \right)$. In particular:
\begin{enumerate}
\item[{\rm (i.1)}] there exist solutions with $\alpha'>0$ and flat instanton, for the Levi-Civita connection
$\nabla^{LC}$ and for any Hermitian connection $\nabla^{\textsf{t}}$ with $\textsf{t} < 0$;
\item[{\rm (i.2)}] for $(\re)=(\frac12,0)$, i.e. for the Bismut connection $\nabla^+$,
the solutions satisfy the heterotic equations of motion with $\alpha'>0$ and flat instanton.
\end{enumerate}
%
%
\item[{\rm (ii)}] For any $(\re)\in \mathbb{R}^2$ such that $\beta(\re) \neq 8$, there exist solutions to the Strominger system
with respect to the connection $\nre$ in the anomaly cancellation condition, with non-flat instanton $A=\nabla^+$
and ${\rm sign}\, (\alpha') = {\rm sign} \left( \beta(\re)-8 \right)$.
In particular, there exist solutions with non-flat instanton and $\alpha'>0$ for any Hermitian connection
$\nabla^{\textsf{t}}$ with $\textsf{t} < -1$.
\end{enumerate}
\end{theorem}

\begin{proof}
Since the torsion is given by $T=JdF_t=-\frac{1}{t}(3\, e^{135} + e^{146} + e^{236} + e^{245})$, we have
$$
dT=-\frac{4}{t^2} (e^{1234} + e^{1256} + e^{3456}).
$$

If we consider the case of flat instanton, then the anomaly cancellation condition with respect to the connection $\nre$ reads as
$$
-\frac{4}{t^2} (e^{1234} + e^{1256} + e^{3456}) = dT = 2 \pi^2 \alpha' p_1(\nre) = - \alpha' \, \frac{\beta(\re)}{2t^4} (e^{1234} + e^{1256} + e^{3456}),
$$
where $\beta(\re)$ is given in \eqref{beta-bis}.
This implies that $\alpha' \beta(\re)=8t^2>0$, so whenever $\beta(\re) \neq 0$ there is a solution
to the Strominger system with
${\rm sign}\, (\alpha') = {\rm sign} \left( \beta(\re) \right)$.

In particular, if $\varepsilon=\rho=0$, i.e. the connection is the Levi-Civita connection, then $\beta(0,0)=1$;
if $\rho= \frac12-\varepsilon$, i.e. the connection is Hermitian,
then $\beta \! \left( \varepsilon, \frac12-\varepsilon \right) = 32\,\varepsilon^2(4\varepsilon-1)$ is positive if and only if $\varepsilon>\frac14$.
This proves (i.1).

For the proof of (i.2), from Proposition~\ref{instan-sl2C-bis}
and \cite{Iv} it follows that the solution for $\nabla^+$ satisfies the heterotic equations of motion.

In order to prove (ii) we will consider the non-flat instanton $A=\nabla^+$ in the anomaly cancellation condition.
Hence, we have to solve
\begin{equation*}
\begin{array}{rl}
-\frac{4}{t^2} (e^{1234} + e^{1256} + e^{3456}) = dT \! & \!\! = 2 \pi^2 \alpha' ( p_1(\nre) - p_1(\nabla^{+})) \\[7pt]
\! & \!\! = - \alpha' \, \frac{\beta(\re)-8}{2t^4} (e^{1234} + e^{1256} + e^{3456}),
\end{array}
\end{equation*}
that is, $\alpha' \left(\beta(\re)-8 \right) =8t^2>0$, so whenever $\beta(\re) \neq 8$ there is a solution with
${\rm sign}\, (\alpha') = {\rm sign} \left( \beta(\re) -8 \right)$.
In particular, for Hermitian connections we have
$$
\beta\!\left(\varepsilon, 1/2-\varepsilon \right) -8 = 8(2\varepsilon-1)(8\varepsilon^2+2\varepsilon+1),
$$
which is positive if and only if $\varepsilon>\frac12$.
\end{proof}

\begin{remark}\label{recover-Fei-Yau}
{\rm
Notice that in the non-flat instanton case, i.e. case (ii) in Theorem~\ref{solutions-sl2C},
there is no invariant solution with $\alpha'>0$ for any of the preferred connections $\nabla^{LC}$, $\nabla^{\pm}$ or $\nabla^{c}$.
In fact, for these connections we have that $\beta(0,0) = 1$, $\beta(\frac12,0) = 8$, $\beta(-\frac12,0) = -4$ and $\beta(0,\frac12) = 0$,
so $\beta-8<0$ in any case.

We also notice that for the particular case of Hermitian connections $\nabla^{\textsf{t}}$, the previous theorem
is in accord with the
results in \cite[Theorem 3.6 and Section~4]{FeiYau}.
In fact, if $\rho= \frac12-\varepsilon$ then we can take $4\,\varepsilon=1-\textsf{t}$, and then
$$
\beta\!\left(\varepsilon, 1/2-\varepsilon \right) -8 = -2 \left[ (1-4\,\varepsilon)(4\,\varepsilon)^2+4 \right]
=-2 \left[ \textsf{t}(\textsf{t}-1)^2+4 \right];
$$
now, by Theorem~\ref{solutions-sl2C}~{\rm (ii)} we conclude that $\alpha'>0$ if and only if
$\textsf{t}(\textsf{t}-1)^2+4<0$, which is precisely the condition given in \cite[page 1194]{FeiYau}
for the non-flat invariant solutions on~$\mathfrak{sl}(2,\mathbb C)$.
}
\end{remark}


\section{The solvmanifold $\frg_7$}\label{g7}

\noindent
In this section we construct many new invariant solutions to the Strominger system on a solvmanifold with respect to the
connections $\nre$ in the anomaly cancellation condition.
In particular, we find solutions for the Chern connection $\nabla^c$ with non-flat instanton and positive $\alpha'$.
Moreover, some solutions satisfy in addition the heterotic equations of motion.

Recall that a solvmanifold is a compact quotient of a simply-connected solvable Lie group $G$ by a lattice $\Gamma$
of maximal rank. As in the previous sections, we will consider \emph{invariant} structures on the solvmanifold.
In \cite[Theorem 2.8]{FOU} the Lie algebras underlying the 6-dimensional solvmanifolds that admit an invariant
complex structure with holomorphically trivial canonical bundle were classified.
One of such Lie algebras is the one denoted by $\frg_7$, whose structure equations are given by
$$
d\beta^1=\beta^{24} + \beta^{35},\ \ d\beta^2=\beta^{46},\ \ d\beta^3=\beta^{56},
\ \ d\beta^4=-\beta^{26},\ \ d\beta^5=-\beta^{36},\ \ d\beta^6=0.
$$
Moreover, the simply-connected solvable Lie group corresponding to $\frg_7$ admits lattices
of maximal rank (see \cite[Proposition 2.10]{FOU} for more details).
Recall that, given a solvmanifold $M=G/\Gamma$, the lattice determines the topology and is actually its fundamental
group. By \cite{Mostow},
two solvmanifolds having isomorphic fundamental groups are diffeomorphic.
Next we will construct invariant solutions on a solvmanifold
with underlying Lie algebra isomorphic to $\frg_7$.

As it is proved in \cite[Proposition 3.6]{FOU}, up to equivalence, there exist only two complex structures on $\frg_7$
such that the associated canonical bundle on the complex solvmanifold is holomorphically trivial.
From now on, we will denote these complex structures by $J_{\delta}$, where $\delta =\pm 1$.
There is a basis $\{ \omega_{\delta}^1, \omega_{\delta}^2, \omega_{\delta}^3 \}$ of forms of type (1,0) with respect to $J_{\delta}$ given by
$$
\omega_{\delta}^1=\beta^4+i\,\beta^2,\quad \omega_{\delta}^2=\beta^3+i\,\beta^5,\quad \omega_{\delta}^3=\frac12 \beta^6+2i\,\delta\beta^1.
$$
Hence, the complex structure equations of $J_{\delta}$ are
\begin{equation}\label{equations2-2-g7-bis}
J_{\delta} \ \colon \quad
d\omega_{\delta}^1\!=i\,\omega_{\delta}^{1}\!\wedge\! (\omega_{\delta}^{3}\!\!+\omega_{\delta}^{\bar{3}}),\quad
d\omega_{\delta}^2\!=\! -i\,\omega_{\delta}^{2}\!\wedge\! (\omega_{\delta}^{3}\!\!+\omega_{\delta}^{\bar{3}}),\quad
d\omega_{\delta}^3\!=\delta\, (\omega_{\delta}^{1\bar{1}}\!\!-\omega_{\delta}^{2\bar{2}}),
\end{equation}
where $\delta =\pm 1$.
Notice that the (3,0)-form $\omega_{\delta}^1\wedge\omega_{\delta}^2\wedge\omega_{\delta}^3$ is closed.

In \cite[Theorem 4.5]{FOU}
it is proved that any balanced metric on $(\frg_7,J_{\delta})$ is given by
\begin{equation}\label{2form-g7}
2\,F^{\delta}_{r,t,u}=i\,(r^2\omega_{\delta}^{1\bar1}+r^2\omega_{\delta}^{2\bar2}+t^2\omega_{\delta}^{3\bar3})+u\,\omega_{\delta}^{1\bar2}-\bar u \,\omega_{\delta}^{2\bar1},
\end{equation}
where $r,t\in \mathbb{R}-\{0\}$ and $u\in \mathbb{C}$ with $r^2-|u|>0$.

Let us consider the real basis of 1-forms $\{e^1,\ldots,e^6\}$ on $\frg_7$ defined as
$$
e^1+i\,e^2 = \frac{\sqrt{r^4-|u|^2}}{r}\, \omega_{\delta}^1, \quad\
e^3+i\,e^4 = \frac{u}{r}\, \omega_{\delta}^1 + i r \, \omega_{\delta}^2, \quad\
e^5+i\,e^6 = t\, \omega_{\delta}^3.
$$
One can check that the basis $\{e^1,\ldots,e^6\}$ is adapted to the balanced structure $(J_{\delta},F^{\delta}_{r,t,u})$. In fact,
the complex structure and the
fundamental form express canonically as
\begin{equation}\label{adapted-basis-g7}
J_{\delta}(e^1)=-e^2,\ J_{\delta}(e^3)=-e^4,\ J_{\delta}(e^5)=-e^6,\quad\quad F^{\delta}_{r,t,u}=e^{12}+e^{34}+e^{56}.
\end{equation}
We consider the holomorphic (3,0)-form $\Psi^{\delta}_{r,t,u}$ given by
\begin{equation}\label{Psi-g7}
\Psi^{\delta}_{r,t,u}=(e^1+i\, e^2)\wedge (e^3+i\, e^4)\wedge (e^5+i\, e^6)=
i\, t\, \sqrt{r^4-|u|^2}\ \omega_{\delta}^{123}.
\end{equation}

Hence, for $\delta=\pm 1$ and for each $r,t\in \mathbb{R}-\{0\}$ and $u\in \mathbb{C}$ such that $r^2>|u|$,
we have an invariant SU(3)-structure $(J_{\delta},F^{\delta}_{r,t,u},\Psi^{\delta}_{r,t,u})$ on any solvmanifold with underlying Lie algebra
isomorphic to $\frg_7$.

Let us write $u=u_1+i\, u_2$. In the adapted basis $\{e^1,\ldots,e^6\}$
one has that the (real) structure equations of $\frg_7$ are
\begin{equation}\label{equations-g7-es}
\left\{\begin{array}{ll}
\!\!&\!\! de^1 =-\frac{2}{t} e^{25},\\[1em]
\!\!&\!\! de^2=\frac{2}{t} e^{15},\\[1em]
\!\!&\!\! de^3= -\frac{4}{t \sqrt{r^4-|u|^2} } (u_2\,e^{15} + u_1\,e^{25}) + \frac{2}{t} e^{45},\\[1em]
\!\!&\!\! de^4= \frac{4}{t \sqrt{r^4-|u|^2} } (u_1\,e^{15} - u_2\,e^{25})- \frac{2}{t}e^{35},\\[1em]
\!\!&\!\! de^5= 0,\\[.8em]
\!\!&\!\! de^6 = -\frac{2 \delta t}{r^2}\,\left[e^{12}-e^{34} - \frac{u_2}{\sqrt{r^4-|u|^2}} (e^{13}+e^{24})
+ \frac{u_1}{\sqrt{r^4-|u|^2}} (e^{14} -e^{23})\right].
\end{array}\right.
\end{equation}

The following result provides instantons that we will use for solving the anomaly cancellation condition.

\begin{proposition}\label{g7familiainstantones}
For each $\lambda,\mu \in \mathbb{R}$, let $A_{\lambda,\mu}$ be the
linear connection defined by the connection $1$-forms
$$
(\sigma^{A_{\lambda,\mu}})^1_2=-(\sigma^{A_{\lambda,\mu}})^2_1=
-(\sigma^{A_{\lambda,\mu}})^3_4=(\sigma^{A_{\lambda,\mu}})^4_3= \lambda\, e^5 + \mu\, e^6,
$$
and $(\sigma^{A_{\lambda,\mu}})^i_j=0$ for any other $(i,j)\not= (1,2), (2,1), (3,4), (4,3)$.
Then, the connection $A_{\lambda,\mu}$ is compatible with any {\rm SU(3)}-structure $(J_{\delta},F^{\delta}_{r,t,u},\Psi^{\delta}_{r,t,u})$,
and it satisfies the instanton condition. Moreover,
the instanton $A_{\lambda,\mu}$ is non-flat if and only if $\mu\not=0$, and
$$
p_1(A_{\lambda,\mu})=-\frac{2\, \mu^2 t^2}{\pi^2(r^4-|u|^2)}\, e^{1234}.
$$
\end{proposition}

\begin{proof}
It is clear that the connection 1-forms satisfy \eqref{A-su3_structure}, so $A_{\lambda,\mu}$ is compatible with any {\rm SU(3)}-structure $(J_{\delta},F^{\delta}_{r,t,u},\Psi^{\delta}_{r,t,u})$.
From \eqref{curvature}, and using the equations \eqref{equations-g7-es}, it follows that the only non-zero curvature forms
$(\Omega^{A_{\lambda,\mu}})^i_j$ of the connection $A_{\lambda,\mu}$
are

\vskip.2cm

$(\Omega^{A_{\lambda,\mu}})^1_2=-(\Omega^{A_{\lambda,\mu}})^2_1=
-(\Omega^{A_{\lambda,\mu}})^3_4=(\Omega^{A_{\lambda,\mu}})^4_3
$

\vskip.3cm

\hskip1.35cm
$
=-\delta\frac{2 \mu t}{r^2}\, (e^{12}-e^{34}) + \delta\frac{2 \mu t u_2}{r^2\sqrt{r^4-|u|^2}}\, (e^{13}+e^{24})
- \delta\frac{2 \mu t u_1}{r^2\sqrt{r^4-|u|^2}}\, (e^{14}-e^{23}).
$

\vskip.2cm

\noindent Now, it is easy to see that $A_{\lambda,\mu}$ satisfies the conditions~\eqref{su3_instanton}, so it is an instanton.
A direct calculation using the previous curvature forms shows that $p_1(A_{\lambda,\mu})$ is given as above.
\end{proof}

Using the instantons found in Proposition~\ref{g7familiainstantones},
our goal is to solve the anomaly cancellation condition
$$
dT= 2 \alpha' \pi^2  (p_1(\nre)-p_1(A_{\lambda,\mu})),
$$
that is to say, we will find which are the connections $\nre$ solving this equation with respect
to some instanton $A_{\lambda,\mu}$ and with $\alpha' \not=0$, putting special
attention to the preferred connections and to the cases when $\alpha'$ is positive and the
instanton $A_{\lambda,\mu}$ is non-flat (i.e. $\mu\not=0$).
We need first to find the expressions of the 4-forms $dT$ and $p_1(\nre)$.

Using the structure equations~\eqref{equations-g7-es} we get that the torsion 3-form $T=J_{\delta}\, dF^{\delta}_{r,t,u}$
is given by
$$
T=
-\frac{2}{r^2}\Big[ \delta\, t\, (e^{12}-e^{34})
-\frac{\delta\, t^2\, u_2 - 2 \,r^2 \,u_1}{t\sqrt{r^4-|u|^2}} (e^{13}+e^{24})
+ \frac{\delta\, t^2\, u_1 + 2 \,r^2 \,u_2}{t\sqrt{r^4-|u|^2}} (e^{14}-e^{23}) \Big] \wedge e^{6}.
$$
Hence, using again \eqref{equations-g7-es}, the 4-form $dT$ is
\begin{equation}\label{dT-g7}
\begin{array}{ll}
dT= \!&\! -\frac{8\, t^2}{r^4-|u|^2} e^{1234}-\frac{16}{t^2\sqrt{r^4-|u|^2}} \left[\frac{2|u|^2}{\sqrt{r^4-|u|^2}}e^{1256} + u_2\, (e^{1356}+e^{2456}) - u_1\, (e^{1456}-e^{2356})\right].
\end{array}
\end{equation}

On the other hand, the curvature 2-forms of the connection $\nre$ are given in Appendix~\ref{apendice3}, and
a long but straightforward calculation allows to find the first Pontrjagin form of $\nre$, which is given by

\begin{equation}\label{p1-nre-g7}
\begin{array}{ll}
p_1(\nre) = \!&\! -\frac{X(\re)\, t^4 + Y(\re)\, |u|^2}{\pi^2 (r^4-|u|^2)^2}\, e^{1234} \\[8pt]
\!&\! -\frac{4\,(Z(\re)\, t^4 + W(\re)\, |u|^2)}{\pi^2 t^4(r^4-|u|^2)^{3/2}} \left[\frac{2|u|^2}{\sqrt{r^4-|u|^2}}e^{1256} + u_2\, (e^{1356}+e^{2456}) - u_1\, (e^{1456}-e^{2356})\right],
\end{array}
\end{equation}

where
\begin{equation}\label{XYZW}
\begin{array}{ll}
\!&\! X(\re) = (1 + 2 \varepsilon - 2 \rho) \left[ 4\varepsilon^2+(1-2\rho)^2+2 \right],\\[8pt]
\!&\! Y(\re) = -4(1 - 2\varepsilon + 2 \rho) \left[4 \varepsilon^2 + (1-2 \rho)^2-2 \right]-8,\\[8pt]
\!&\! Z(\re) = (1 +\varepsilon - \rho) \left[4 \varepsilon^2 + (1-2 \rho)^2-4 \right]+3,\\[8pt]
\!&\! W(\re) = 4 (\varepsilon - \rho) \left[4 \varepsilon^2 + (1-2 \rho)^2 \right].
\end{array}
\end{equation}

From the expressions of $dT$, $p_1(\nre)$ and $p_1(A_{\lambda,\mu})$, respectively given by \eqref{dT-g7}, \eqref{p1-nre-g7} and
Proposition~\ref{g7familiainstantones}, the anomaly cancellation condition reduces to
the following system of two equations
$$
\left\{
\begin{array}{l}
\frac{8\, t^2}{r^4-|u|^2} = \frac{\alpha'}{4} \left( \frac{8}{(r^4-|u|^2)^2}(X(\re)\, t^4 + Y(\re)\, |u|^2)
- \frac{16\, \mu^2 t^2}{r^4-|u|^2} \right),\\[10pt]
\frac{32\, |u|^2}{t^2(r^4-|u|^2)} = \frac{\alpha'}{4} \frac{64\, |u|^2}{t^4(r^4-|u|^2)^2}\, (Z(\re)\, t^4 + W(\re)\, |u|^2) ,
\end{array}
\right.
$$
which is equivalent to
\begin{equation}\label{ecus-anomaly-g7}
\left\{
\begin{array}{l}
4\, t^2 (r^4-|u|^2) = \alpha' \, \Big( X(\re)\, t^4 + Y(\re)\, |u|^2 -2\, \mu^2t^2(r^4-|u|^2) \Big),\\[9pt]
2\, |u|^2 t^2 (r^4-|u|^2) = \alpha' \, |u|^2 \Big( Z(\re)\, t^4 + W(\re)\, |u|^2 \Big).
\end{array}
\right.
\end{equation}

\vskip.2cm

In view of the second equality in \eqref{ecus-anomaly-g7}, we will distinguish two cases depending on the vanishing of
the (metric) coefficient $u$.
Another important reason to distinguish the cases $u=0$ and $u\not=0$, is the following result, which proves that the connection $\nre$
satisfies the instanton condition
if and only if $\nre=\nabla^{+}$ and $u=0$.
We will use this fact in Section~\ref{u0} to provide new solutions to the heterotic equations of motion.

\begin{proposition}\label{g7nabla+instanton}
The connection $\nre$ is an instanton with respect to an {\rm SU(3)}-structure $(J_{\delta},F^{\delta}_{r,t,u},\Psi^{\delta}_{r,t,u})$
if and only if $u=0$ and $(\re)=(\frac12,0)$, i.e. the metric coefficient $u$ vanishes and $\nre$ is the Bismut connection $\nabla^{+}$.
Furthermore,
if $u=0$ then the Bismut connection satisfies
$$
p_1(\nabla^{+})=-\frac{8\,t^4}{\pi^2r^8}\,e^{1234}.
$$
\end{proposition}

\begin{proof}
Let $(\Omega^{\varepsilon,\rho})_j^i$ be the curvature 2-forms of the connection $\nre$, which are given in Appendix~\ref{apendice3}.
We impose the conditions~\eqref{su3_instanton}.
From
$$
(\Omega^{\varepsilon,\rho})^1_5(J_{\delta}e_1,J_{\delta}e_6)-(\Omega^{\varepsilon,\rho})^1_5(e_1,e_6)=
2\delta\, \frac{(1+2\varepsilon-2\rho)|u|^2}{r^2(r^4-|u|^2)}=0
$$
we get that $(1+2\varepsilon-2\rho)u=0$.
If $1+2\varepsilon-2\rho=0$, then the condition $(\Omega^{\varepsilon,\rho})^1_3(J_{\delta}e_1,J_{\delta}e_3)=(\Omega^{\varepsilon,\rho})^1_3(e_1,e_3)$
is equivalent to $\varepsilon=0$, but in this case
$(\Omega^{\varepsilon,\rho})^5_6(e_1,e_2)+(\Omega^{\varepsilon,\rho})^5_6(e_3,e_4)+(\Omega^{\varepsilon,\rho})^5_6(e_5,e_6)=
-4t^2/(r^4-|u|^2)$,
which never vanishes.

Therefore, we must have $1+2\varepsilon-2\rho\neq0$ and $u=0$. Now, the conditions
$$
\begin{array}{l}
(\Omega^{\varepsilon,\rho})^1_5(J_{\delta}e_1,J_{\delta}e_5)-(\Omega^{\varepsilon,\rho})^1_5(e_1,e_5)=
\frac{-2\rho(1+2\varepsilon-2\rho)t^2}{r^4}=0,\\[8pt]
(\Omega^{\varepsilon,\rho})^1_6(J_{\delta}e_1,J_{\delta}e_6)-(\Omega^{\varepsilon,\rho})^1_6(e_1,e_6)=
\frac{(2\varepsilon-1)(1+2\varepsilon-2\rho)t^2}{r^4}=0,
\end{array}
$$
are satisfied if and only if $(\re)=(\frac12,0)$.

One can check that for $\varepsilon=\frac12$ and $\rho=0$ all the conditions in~\eqref{su3_instanton} are satisfied,
so the Bismut connection $\nabla^{+}$ is an instanton when the metric coefficient $u$ is zero.
Finally, for $u=0$ the Pontrjagin form $p_1(\nabla^{+})$ comes directly from \eqref{p1-nre-g7} by taking $\varepsilon=\frac12$ and $\rho=0$.
\end{proof}

\subsection{The case $u=0$.}\label{u0}

Here we consider the SU(3)-structures given by~\eqref{adapted-basis-g7}-\eqref{Psi-g7} with $u=0$, that is,  $(J_{\delta},F^{\delta}_{r,t,0},\Psi^{\delta}_{r,t,0})$.
In this case the anomaly cancellation condition is reduced just to solve the first equation in \eqref{ecus-anomaly-g7} for $u=0$, i.e.
$$
4 t^2 r^4 = \alpha' (X(\re)\, t^4 -2\, \mu^2t^2r^4),
$$
or equivalently,
$$
 (X(\re)\, t^2 -2\, \mu^2 r^4)\, \alpha' = 4 r^4 >0.
$$
Notice that the system does not depend on $\delta$, so it is the same for the complex structures $J_{-1}$ and $J_{+1}$.
Notice also that ${\rm sign}\, (\alpha') = {\rm sign} \left( X(\re)\, t^2 -2\, \mu^2 r^4 \right)$.
It is clear that $\alpha'$ cannot be positive if $X(\re)\leq 0$.
In that case, we can choose a non-flat instanton (i.e. $\mu\not=0$) to solve the Strominger system with $\alpha'<0$.
It follows from \eqref{XYZW} that $X(\re)\leq 0$ if and only if $\rho \geq \varepsilon + \frac12$.
For instance, $\nabla^-$ and $\nabla^c$ are connections in this case.

The connections $\nre$ for which there is a solution to the anomaly cancellation condition with
$\alpha'>0$ are precisely those satisfying $X(\re)>0$, i.e. $\rho < \varepsilon + \frac12$.
Actually, we can choose a non-flat instanton with sufficiently small $\mu\not=0$.
In particular, $\nabla^{LC}$ and $\nabla^+$ are connections in this case.

Moreover, in the case of the Bismut connection $\nabla^+$ the solutions solve in addition the heterotic equations of motion
because $\nabla^+$ satisfies the instanton condition by Proposition~\ref{g7nabla+instanton}.

Therefore, we have proved the following result:

\begin{theorem}\label{solutions-g7-u=0}
Let us consider a solvmanifold with underlying Lie algebra $\frg_7$ endowed with an {\rm SU(3)}-structure
given by~\eqref{adapted-basis-g7}-\eqref{Psi-g7} with $u=0$. Then, the Strominger system has invariant solutions
for any connection $\nre$ and
with a non-flat instanton $A_{\lambda,\mu}$.
More concretely:
\begin{enumerate}
\item[{\rm (i)}]
If $\rho \geq \varepsilon +\frac12$, then there exist solutions to the Strominger system
with respect to the connection $\nre$ in the anomaly cancellation condition, with $\alpha'<0$ and non-flat instanton.
In particular, there are solutions
with respect to
the connection $\nabla^-$
and with respect to the Hermitian connection $\nabla^{\textsf{t}}$ for any $\textsf{t} \geq 1$,
which includes the Chern connection $\nabla^c$.
\item[{\rm (ii)}]
If $\rho < \varepsilon +\frac12$, then there exist solutions to the Strominger system
with respect to the connection $\nre$ in the anomaly cancellation condition, with $\alpha'>0$ and non-flat instanton.
In particular, there are solutions
with respect to the Levi-Civita connection $\nabla^{LC}$
and with respect to the Hermitian connection $\nabla^{\textsf{t}}$ for any $\textsf{t} < 1$,
which includes the Bismut connection $\nabla^+$.
\item[{\rm (iii)}]
Moreover, for the Bismut connection the solutions satisfy in addition the heterotic equations of motion with $\alpha'>0$
and non-flat instanton.
\end{enumerate}
\end{theorem}

Notice that the solutions are given by Figure~\ref{conexiones-3}, so the Chern connection $\nabla^c$
is excluded again from the $\alpha'>0$ case.
In order to find solutions to the Strominger system with $\alpha'>0$ and
with respect to $\nabla^c$ in the anomaly cancellation condition, we will need a detailed study of the
case $u\not=0$ below.

\subsection{The case $u\not=0$.}\label{unot0}

Here we consider the ansatz $(J_{\delta},F^{\delta}_{r,t,u},\Psi^{\delta}_{r,t,u})$ for $u\not=0$.
We have to solve the system \eqref{ecus-anomaly-g7}, that is,
$$
\begin{array}{ll}
\alpha' \left( X(\re)\, t^4 + Y(\re)\, |u|^2 -2\, \mu^2t^2(r^4-|u|^2) \right)
\!\!&\!\! = 2\alpha' (Z(\re)\, t^4 + W(\re)\, |u|^2)\\[8pt]
\!\!&\!\! = 4 t^2 (r^4-|u|^2).
\end{array}
$$
Since $\alpha' \not=0$ and $4 t^2 (r^4-|u|^2)>0$, the sign of $\alpha'$ must be the same as the sign of $Z(\re)\, t^4 + W(\re)\, |u|^2$.
Notice that the system does not depend on $\delta$.

From the first equality above we get
$$
\begin{array}{ll}
0 \leq 2\, \mu^2t^2(r^4-|u|^2) \!\!&\!\! = X(\re)\, t^4 + Y(\re)\, |u|^2 - 2 (Z(\re)\, t^4 + W(\re)\, |u|^2).
\end{array}
$$
Observe that we can rewrite the previous expression as
$$0 \leq 2\, \mu^2t^2(r^4-|u|^2)= -4(L(\re)\, t^4 + N(\re)\, |u|^2), $$
where $L$ and $N$ are given by
$$
\begin{array}{ll}
\!&\! X(\re) - 2 \, Z(\re) = -4 \left[ (\varepsilon-\frac{3}{2})^2+(\rho+1)^2-4 \right] = -4\, L(\re),
\\[8pt]
\!&\! Y(\re) - 2 \, W(\re) = -16 \left[ (\varepsilon+\frac{1}{2})^2+(\rho-1)^2-1 \right] = -4 \, N(\re).
\end{array}
$$

Equation $L(\re)=0$ represents a circle with center at the point $(\frac32,-1)$
and of radius~$2$, whereas $N(\re)=0$ represents a circle with center at the point $(-\frac12,1)$
and of radius~$1$.
These two circles intersect at the points $P_{1}=\left(\frac{1-\sqrt{7}}{8},\frac{3-\sqrt{7}}{8}\right)$
and
$Q_{1}=\left(\frac{1+\sqrt{7}}{8},\frac{3+\sqrt{7}}{8}\right)$
(see Figure~\ref{figura}).

\medskip

Therefore, we have to find the connections $\nre$ for which there exist
$r,t\in \mathbb{R}-\{0\}$ and $u\in \mathbb{C}-\{0\}$ with $r^4-|u|^2>0$
(i.e. an SU(3)-structure given by~\eqref{adapted-basis-g7}-\eqref{Psi-g7} with $u\not=0$)
satisfying
\begin{equation}\label{sistema}
\left\{
\begin{array}{l}
L(\re)\, t^4 + N(\re)\, |u|^2 \leq 0,\\[4pt]
Z(\re)\, t^4 + W(\re)\, |u|^2 \not= 0.
\end{array}
\right.
\end{equation}
In fact, once a solution of \eqref{sistema} is found, we take an instanton given by Proposition~\ref{g7familiainstantones}
with $\mu = \sqrt{\frac{-2(L(\re) \, t^4 + N(\re) \, |u|^2)}{t^2(r^4-|u|^2)}}$,
and the Strominger system is then solved with
$\alpha' = \frac{2\, t^2 (r^4-|u|^2)}{Z(\re)\, t^4 + W(\re)\, |u|^2}$.
The instanton will be non-flat if and only if $\mu\not=0$, i.e. if and only if $L(\re)\, t^4 + N(\re)\, |u|^2 < 0$.

In particular, if we want solutions with non-flat instanton and $\alpha'>0$, then we need to solve the following system:
\begin{equation}\label{sistema-positivo}
\left\{
\begin{array}{l}
L(\re)\, t^4 + N(\re)\, |u|^2 < 0,\\[4pt]
Z(\re)\, t^4 + W(\re)\, |u|^2 > 0.
\end{array}
\right.
\end{equation}

\vskip.2cm

Next we study some distinguished regions in the $(\re)$-plane related to the system~\eqref{sistema-positivo}
that will play a central role.

\vskip.2cm

\begin{notation}\label{notacion}
In what follows, we will use the following notation:
$$L^- = \{ (\re) \in \mathbb{R}^2 \mid L(\re)<0 \},\quad\quad \overline{L^-} = \{ (\re) \in \mathbb{R}^2 \mid L(\re)\leq 0 \},
$$
$$L^+ = \{ (\re) \in \mathbb{R}^2 \mid L(\re)>0 \},\quad\quad \overline{L^+} = \{ (\re) \in \mathbb{R}^2 \mid L(\re)\geq 0 \}.
$$
Similar notation will apply to $N$ and to any other real function defined in the $(\re)$-plane.
In particular, $L^-$ and $N^-$ correspond to the interior of the circles $L(\re)=0$ and $N(\re)=0$, respectively.
\end{notation}

In order to determine the sign of $\alpha'$ in the solutions, we will also need to study the sign of the determinant
associated to \eqref{sistema}:
\begin{eqnarray*}
d(\re)=\det\begin{pmatrix}
L(\re)&N(\re)\\[5pt]
Z(\re)&W(\re)
\end{pmatrix}
&\!\!\!=\!\!\!&
-4(\varepsilon^2+(\rho-1/2)^2-1/4) (20 \varepsilon^2 + 20 \rho^2 - 32 \varepsilon \rho + 4 \varepsilon - 8 \rho +1)\\
&\!\!\!=\!\!\!&-4\, M(\re)\, S(\re),
\end{eqnarray*}
where $M(\re)=\varepsilon^2+(\rho-1/2)^2-1/4$,
and $S(\re)=20 \varepsilon^2 + 20 \rho^2 - 32 \varepsilon \rho + 4 \varepsilon - 8 \rho +1$.

The equation $M(\re)=0$ represents a circle with center at the point $(0,\frac12)$,
i.e. the center is the Chern connection, and of radius $1/2$.
Notice that this circle passes through the points $P_1$ and $Q_1$ obtained above, and also
through the points $P_2=(0,0)$ and $Q_2=(\frac12,\frac12)$, which
will play a role below.
On the other hand, $S(\re)$ can be rewritten as
$$
S(\re) = 20\left(\varepsilon + \frac{1-8\rho}{10}\right)^2 + \frac{36}{5}\left(\rho-\frac13\right)^2.
$$
It is clear that $S(\re)\geq 0$, and $S(\re)=0$ if and only if $(\re)=P_3=(\frac16,\frac13)$.
Moreover, since  $M(P_3)<0$ we conclude that
\begin{itemize}
\item $d(\re) = 0$ if and only if $(\re) = P_3$ or $(\re)$ lies on the circle $M(\re)=0$;
\item $d(\re)>0$ if and only if $(\re) \in M^- -\{P_3\}$;
\item $d(\re)<0$ if and only if $(\re) \in M^+$.
\end{itemize}

\bigskip

Now we consider the second condition in \eqref{sistema}. It is easy to verify that the intersection
of $Z(\re)= 0$ and $W(\re)= 0$ is given by the two points $P_{2}$ and $Q_{2}$ given above.
These points belong to $L^-$.
Let us denote by $\Delta$ the following region in the $(\re)$-plane (see Figure~\ref{figura}):
$$
\Delta=L^- \cup N^- - \{ P_2,Q_2 \}.
$$

Let $\Delta_+ \subset \Delta$ be the region
(see Figure~\ref{figura}) defined by
$$
\Delta_+=L^- \cup M^- - (\overline{M^-} \cap \overline{Z^-}),
$$
where we are using the notation introduced in Notation~\ref{notacion}.
Notice that $\Delta_+$ and $\Delta$ are connected subsets of the plane.

\medskip

\begin{figure}[h]
\begin{center}
\includegraphics[width=8cm]{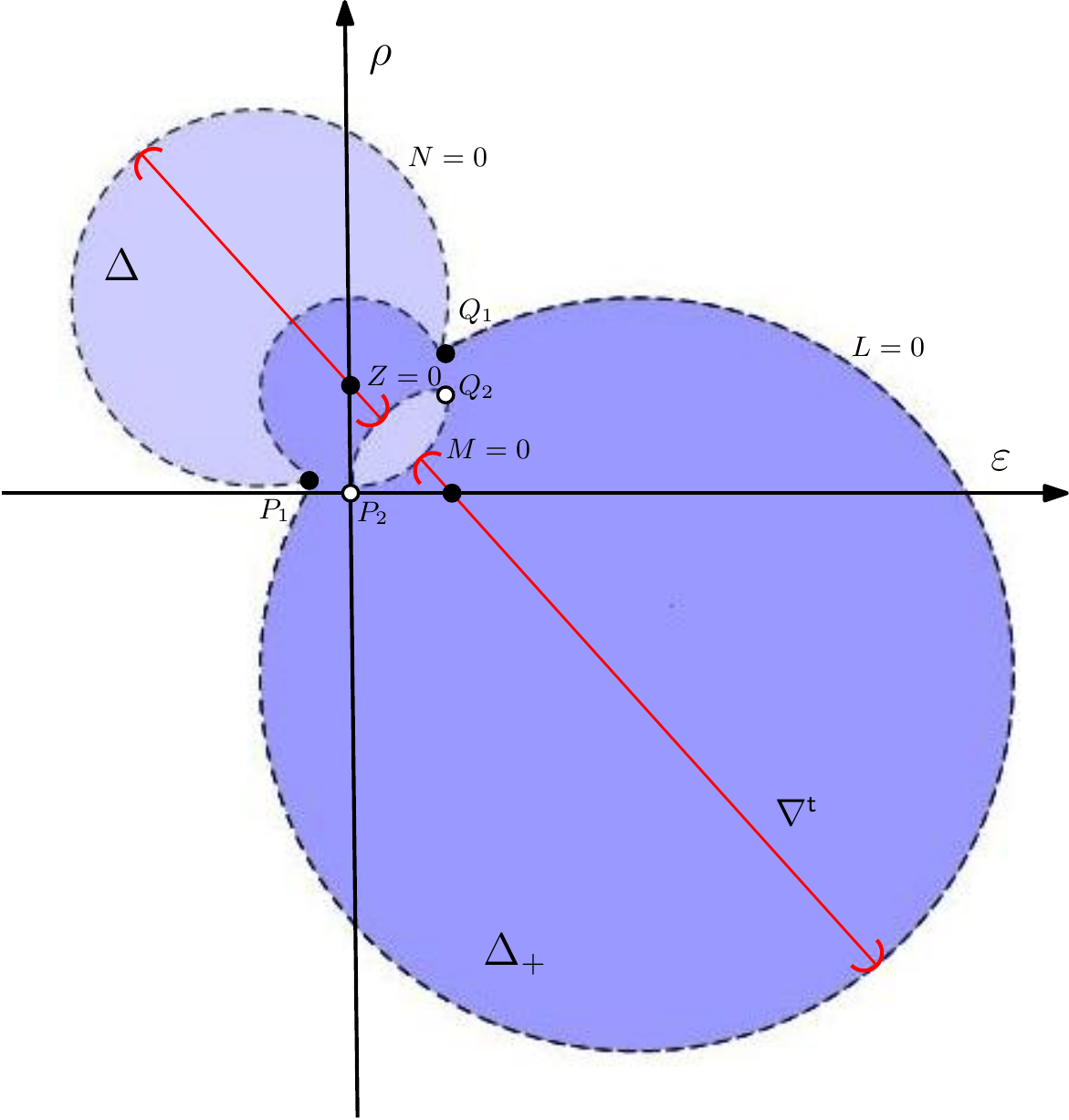}\caption{Connections $\nre$ solving the Strominger system
with non-flat instanton in the solvmanifold case according to Theorem~\ref{solutions-g7-unot0};
the region $\Delta_+$ corresponds to solutions with $\alpha'>0$,
which in particular
includes the Chern connection.}\label{figura}
\end{center}
\end{figure}

\medskip

\bigskip

Now we are in a position to present the main result in this section.

\begin{theorem}\label{solutions-g7-unot0}
Let us consider a solvmanifold with underlying Lie algebra $\frg_7$.
There are {\rm SU(3)}-structures
given by~\eqref{adapted-basis-g7}-\eqref{Psi-g7} with $u\neq0$
providing solutions to the Strominger system with respect to a
connection $\nre$ in the anomaly cancellation condition,
in the following cases:
\begin{enumerate}
\item[{\rm (i)}] For $(\re) \in \{P_1,Q_1\}$, there exist solutions with $\alpha'>0$ and flat instanton.
\item[{\rm (ii)}] For any $(\re) \in \Delta$,
there exist solutions to the Strominger system with non-flat instanton. Moreover,
$\alpha'>0$ if and only if $(\re) \in \Delta_+$.
\end{enumerate}
\end{theorem}

\begin{proof}
The proof of (i) is clear, because if $(\re)=P_1$ or $Q_1$, then $L(\re)=0=N(\re)$ and the first equation in~\eqref{sistema}
is trivially satisfied. Notice that $\mu=0$, i.e. the instanton given by Proposition~\ref{g7familiainstantones} is flat.
Since $Z(P_1)=Z(Q_1)=\frac34>0$ and $W(P_1)=W(Q_1)=-1<0$, we can take any metric whose coefficient $t$ is large enough so that
the second condition in~\eqref{sistema}
is satisfied with positive sign and therefore $\alpha'>0$.

For the proof of (ii), we first notice that
the first inequality in \eqref{sistema} can be always solved with positive sign
if $(\re)$ belongs to the interior of the figure determined by the two circles,
i.e. $L^-\cup N^-$.
In fact, if for instance $L(\re)<0$, then we can choose $t$ sufficiently large
so that the inequality is solved; similarly for the case $N(\re)<0$ by choosing $u\not=0$
with sufficiently large $|u|$.
Since $\mu\not=0$, the instanton given in Proposition~\ref{g7familiainstantones} is non-flat.
Moreover, since the points $P_{2}$ and $Q_{2}$ of intersection of $Z(\re)= 0$ and $W(\re)= 0$
do not belong to region $\Delta$, we have that the second condition in \eqref{sistema} can be solved,
possibly after a small perturbation of the values of $t$ and $u$
found for the first condition in \eqref{sistema}. This ensures that $\alpha'\not=0$ and the first part of (ii)
is proved.

\medskip

It remains to study when $\alpha'>0$, i.e. which are the values of $(\re) \in \Delta$
for which the conditions in~\eqref{sistema-positivo} are satisfied.
Our discussion is based on $Z(\re)$, and we distinguish the cases $Z^+$, $Z^-$ and $Z=0$.

\smallskip

Let us suppose first that $(\re)\in \Delta$ belongs to $Z^+$:
\begin{itemize}
\item If $(\re)\in Z^+ \cap L^-$, then it is enough to take a sufficiently large coefficient $t$ in the metric in order to solve \eqref{sistema-positivo}.
\item If $(\re)\in Z^+ \cap L^+$, then necessarily $N(\re)<0$. Moreover, the conditions in \eqref{sistema-positivo} imply
$$
-N(\re)\, |u|^2 > L(\re)\, t^4 > \frac{L(\re)(-W(\re))}{Z(\re)}\, |u|^2,
$$
that is, the determinant $d(\re)>0$.
\item Finally, if $Z(\re) > 0$ and $L(\re) = 0$, then necessarily $N(\re)<0$ and $(\re)$ belongs to the arc joining $P_1$ and $Q_1$ along
the circle $L(\re) = 0$ inside $N^-$.  In this case, the determinant $d(\re)$ is positive.
\end{itemize}
In conclusion, if $Z(\re) > 0$, then there exist solutions to~\eqref{sistema-positivo} if
and only if $(\re)\in L^-\cup d^+$.

\medskip

On the other hand, if $(\re)\in \Delta$ belongs to $Z^-$, then necessarily $W(\re)>0$, which is equivalent to say that $\varepsilon > \rho$.
Hence:
\begin{itemize}
\item If $(\re)\in Z^- \cap \overline{L^+}$, then $(\re)$ must lie in the interior of the circle $N=0$ but in this case $W(\re)$ is negative, which is a contradiction.
\item If $(\re)\in Z^- \cap {L^-}$, then arguing as before we can see that $d(\re)<0$.
\end{itemize}
In conclusion, if  $Z(\re) < 0$, then there exist solutions to~\eqref{sistema-positivo} if and only if $(\re)\in L^-\cap d^-$.

\medskip

Finally, if $(\re)$ satisfies $Z(\re)=0$, then necessarily $W(\re)>0$.
Moreover, ${\rm sign}\, d(\re) = {\rm sign}\, L(\re)$.
As before, the case $\overline{L^+}$ leads to a contradiction.

\medskip

As a consequence of the previous discussion we get that the connections $\nre$
for which there exist solutions with $\alpha'>0$ and non-flat instanton (i.e. $\mu\not=0$),
are those for which $(\re)\in (L^- \cup d^+) - (\overline{d^+} \cap \overline{Z^-})$.
Since $Z(P_3)>0$, the latter is precisely the region $\Delta_+$.
\end{proof}

As a consequence of the previous theorem (see also Figure~\ref{figura}) one obtains solutions for many Hermitian connections,
including the Chern connection.
More precisely, we have:

\begin{corollary}\label{solutions-g7-unot0-cor1}
A solvmanifold with underlying Lie algebra $\frg_7$ provides invariant
solutions to the Strominger system with $\alpha'>0$ and non-flat instanton with respect to a
Hermitian connection $\nabla^{\textsf{t}}$
for
$\textsf{t} \in (-5-4\sqrt{2},1-\sqrt{2}) \cup (\frac{5-\sqrt{17}}{2},1+\sqrt{2})$.
In particular, there are solutions for the Bismut connection ($\textsf{t}=-1$)
and for the Chern connection ($\textsf{t}=1$).
\end{corollary}

\begin{proof}
The result follows from Theorem~\ref{solutions-g7-unot0}~(ii) by taking
the intersection of the line $\rho+\varepsilon=\frac12$ with the region $\Delta_+$.
A direct calculation shows that the corresponding values of $\varepsilon$ are
$\varepsilon\in\left(-\frac{\sqrt{2}}{4}, \frac{-3+\sqrt{17}}{8}\right)
\cup \left(\frac{\sqrt{2}}{4}, \frac{3+2\sqrt{2}}{2}\right)$.
Since $\textsf{t}=1-4\,\varepsilon$, we get the result.
\end{proof}

\begin{remark}\label{conclusion}
{\rm
Notice that in \cite{UV1} it is proved that the nilmanifold $\frh_{19}^-$
has invariant solutions with respect to the Chern connection with positive $\alpha'$ and non-flat instanton.
There do not exist such solutions on the other balanced nilmanifolds.
In addition, $\frh_{19}^-$ has a solution with respect to the Bismut connection \cite{FIUV09}.
However, the solutions on the nilmanifold $\frh_{19}^-$ never
satisfy the heterotic equations of motion.

In conclusion, as far as we know, the solvmanifold $\frg_7$ is the first known example of a compact
complex manifold that provides explicit solutions to the Strominger system, with $\alpha'>0$ and non-flat instanton,
both with respect to the Chern connection and to the Bismut connection, the latter being also solutions
to the heterotic equations of motion.
}
\end{remark}

In \cite[Proposition 5.7]{UV2} it is proved that the nilmanifold $\frh_{19}^-$ has
a balanced Hermitian structure which provides simultaneously solutions to
the Strominger systems for $\nabla^+$ and for $\nabla^c$.
As a consequence of Theorem~\ref{solutions-g7-unot0}
we prove a similar result for solvmanifolds with underlying Lie algebra $\frg_7$.
This makes the space of solutions on these solvmanifolds even richer.

\begin{corollary}\label{solutions-g7-unot0-cor2}
Let us consider a solvmanifold with underlying Lie algebra $\frg_7$.
There is an {\rm SU(3)}-structure
and a non-flat instanton solving at the same time the Strominger systems
for the Bismut and for the Chern connection, both with positive $\alpha'$'s.
\end{corollary}

\begin{proof}
Let us describe first the whole space of solutions to the Strominger system for the Chern connection
with $\alpha'>0$ and with non-flat instanton.
For $(\re)=(0,\frac12)$ one has that $L(0,\frac12)=\frac12>0$, $N(0,\frac12)=-2<0$, $Z(0,\frac12)=1>0$ and $W(0,\frac12)=0$, therefore
for $(\re)=(0,\frac12)$ the system \eqref{sistema-positivo} is equivalent to $2|u| > t^2 >0$.
Hence, the SU(3)-structures for which the Strominger system has this kind of solutions
are given by~\eqref{adapted-basis-g7}-\eqref{Psi-g7} with $r,t,u$ satisfying
\begin{equation}\label{sistema-positivo-Chern}
2\, r^2 > 2\, |u| > t^2 >0.
\end{equation}
Notice that the instanton is given by Proposition~\ref{g7familiainstantones}
with $\mu=\sqrt{\frac{4|u|^2-t^4}{t^2(r^4-|u|^2)}}\not=0$,
and $\alpha'=\frac{2(r^4-|u|^2)}{t^2}>0$.

Now, we describe all the solutions to the Strominger system for the Bismut connection
with $\alpha'>0$ and with non-flat instanton.
For $(\re)=(\frac12,0)$ one has $L(\frac12,0)=-2<0$, $N(\frac12,0)=4>0$, $Z(\frac12,0)=0$ and $W(\frac12,0)=4>0$, therefore
for $(\re)=(\frac12,0)$ the system \eqref{sistema-positivo} is equivalent to $t^2 > \sqrt{2}\, |u| > 0$.
Hence, the SU(3)-structures for which the Strominger system has this kind of solutions
are given by~\eqref{adapted-basis-g7}-\eqref{Psi-g7} with $r,t,u$ satisfying
\begin{equation}\label{sistema-positivo-Bismut}
t^2 > \sqrt{2}\, |u| > 0,\quad\quad
r^2  > |u|>0.
\end{equation}
The instanton is given by Proposition~\ref{g7familiainstantones}
with $\tilde{\mu}=2 \sqrt{\frac{t^4-2 \, |u|^2}{t^2(r^4-|u|^2)}}\not=0$.
Note that $\tilde{\alpha}'=\frac{t^2 (r^4-|u|^2)}{2\, |u|^2}>0$.

Hence, the existence of an SU(3)-structure and a non-flat instanton solving at the same time the Strominger systems
for the Bismut and for the Chern connection is equivalent to the existence of $r,t,u$
satisfying \eqref{sistema-positivo-Chern}, \eqref{sistema-positivo-Bismut}
and $\mu=\tilde{\mu}$. The latter is equivalent to $t^2=\frac{2\sqrt{3}}{\sqrt{5}}\, |u|$.
Since $\frac{2\sqrt{3}}{\sqrt{5}} \in (\sqrt{2},\ 2)$, we get the desired solution.
Notice that $\alpha'\not=\tilde{\alpha}'$.
\end{proof}

\section{Holonomy and cohomological properties of the solutions}\label{holonomy-cohomology}

\noindent
In this section we determine the holonomy group of the Bismut connection $\nabla^{+}$ of the balanced Hermitian metrics $F$ considered
in the previous sections. We also study some properties of the de Rham cohomology class defined by the 4-form $F^2$.

In \cite{UV2} it is proved that for any invariant balanced Hermitian metric on a nilmanifold with $\frh_3$ as underlying
Lie algebra, the holonomy group of the associated Bismut connection
reduces to the subgroup U(1) of SU(3). For the solvmanifolds found in the previous section we have:

\begin{proposition}\label{hol-g7}
Let us consider a solvmanifold with underlying Lie algebra $\frg_7$, endowed with a balanced Hermitian metric $F_{r,t,u}$
given by~\eqref{2form-g7}.
Then, the holonomy of the associated Bismut connection is {\rm SU(3)} if $u\neq 0$, and reduces to {\rm U(1)} when $u=0$.
\end{proposition}

\begin{proof}
We apply the Ambrose-Singer theorem, for which we need to consider the
curvature endomorphisms of the connection $\nabla^+$.
The curvature endomorphisms $R^+(e_p,e_q)$ are given in terms of the curvature forms~$(\Omega^+)^i_j$ by
\begin{equation}\label{endo-curv}
g(R^+(e_p,e_q)e_i,e_j) = - (\Omega^+)^i_j (e_p,e_q).
\end{equation}
Here $\{e_1,\ldots,e_6\}$ is the dual of any basis $\{e^1,\ldots,e^6\}$ adapted to the balanced structure.
By Appendix~\ref{apendice3} for $\varepsilon=\frac12$ and $\rho=0$, we obtain
that $(\Omega^{+})^1_2+(\Omega^{+})^3_4+(\Omega^{+})^5_6=0$,
so the holonomy group of the Bismut connection $\nabla^+$ reduces to a subgroup of SU(3).

On the other hand, we consider the following curvature endomorphisms:
$$
\begin{array}{lll}
R^+(e_1,e_5) \!\!&\!\!=\!\!&\!\! \frac{-4}{t^2\sqrt{r^4-|u|^2}}\left[\frac{3|u|^2}{\sqrt{r^4-|u|^2}}(e^{15}+e^{26}) - 2u_1(e^{35} + e^{46}) + 2u_2 (e^{36}-e^{45}) \right],\\[10pt]
R^+(e_2,e_5) \!\!&\!\!=\!\!&\!\! \frac{4}{t^2\sqrt{r^4-|u|^2}}\left[\frac{3|u|^2}{\sqrt{r^4-|u|^2}}(e^{16}-e^{25}) - 2u_2(e^{35} + e^{46}) - 2u_1 (e^{36}-e^{45}) \right],\\[10pt]
R^+(e_3,e_5) \!\!&\!\!=\!\!&\!\! \frac{4}{t^2\sqrt{r^4-|u|^2}}\left[ 2u_1(e^{15} + e^{26}) + 2u_2 (e^{16}-e^{25}) + \frac{|u|^2}{\sqrt{r^4-|u|^2}}(e^{35}+e^{46})  \right],\\[10pt]
R^+(e_4,e_5) \!\!&\!\!=\!\!&\!\! \frac{4}{t^2\sqrt{r^4-|u|^2}}\left[ 2u_2(e^{15} + e^{26}) - 2u_1 (e^{16}-e^{25}) - \frac{|u|^2}{\sqrt{r^4-|u|^2}}(e^{36}-e^{45})  \right].
\end{array}
$$
A direct calculation shows that if $|u|^2\frac{(r^4-|u|^2)+3r^4}{(r^4-|u|^2)^2}\neq 0$, then
$R^+(e_1,e_5)$, $R^+(e_2,e_5)$, $R^+(e_3,e_5)$ and $R^+(e_4,e_5)$ are independent.
Since $r^2>|u|$, this happens if and only if the metric coefficient $u\neq 0$.
Similarly, one can check that if $u\neq 0$ then the curvature endomorphisms $R^+(e_1,e_2)$,
$R^+(e_1,e_3)$, $R^+(e_1,e_4)$ and $R^+(e_3,e_4)$ are independent from the previous ones.
Hence, the space generated by the curvature endomorphisms has dimension 8.
This implies that the holonomy of the Bismut connection is SU(3) for any
balanced metric given by ~\eqref{2form-g7} with $u\neq 0$.

Finally, if $u=0$ then from Appendix~\ref{apendice3} and~\eqref{endo-curv} it is easy to check that
the only curvature endomorphisms $R^+(e_p,e_q)$ that do not vanish
are $R^+(e_1,e_2)=-R^+(e_3,e_4)=-\frac{4 t^2}{r^4} (e^{12}-e^{34})$.
Moreover, any covariant derivative $\nabla^+_{e_j}(e^{12}-e^{34})$ of the 2-form $e^{12}-e^{34}$ is either zero or a multiple of itself.
Hence, the holonomy of $\nabla^+$ reduces to U(1) for any balanced metric~\eqref{2form-g7} with $u=0$.
\end{proof}

\begin{remark}\label{remark-hol-g7}
{\rm
It is worth comparing the result in Proposition~\ref{hol-g7} with previous results on
holonomy reduction obtained in~\cite{AV} and~\cite{UV2}:

\begin{itemize}
\item In~\cite{UV2} it is proved that for any invariant balanced structure $(J,F)$ on a 6-dimensional nilmanifold, the holonomy
of the associated Bismut connection equals SU(3) if and only if the complex structure $J$ is not of Abelian type.
(Recall that $J$ is called Abelian if it satisfies the condition $[Jx,Jy]=[x,y]$ for all $x,y \in \frg$.)
Moreover, the holonomy reduces to a subgroup of SU(2) when $J$ is Abelian, and it is equal to U(1) when the Lie algebra
is precisely $\frh_3$. In other words, on 6-dimensional balanced nilmanifolds the holonomy
of the Bismut connection does not depend on the metric, since it is completely determined by the type of the complex structure.
In contrast, Proposition~\ref{hol-g7} shows that this is no longer true in the solvable case, since one gets different holonomy
depending on the metric coefficient $u$.

\item Let $\frg$ be a $2n$-dimensional unimodular Lie algebra equipped with a balanced structure $(J,F)$.
In~\cite{AV} it is proved that if $J$ is Abelian, then
the holonomy group of the associated Bismut connection reduces to a subgroup of SU($n-k$),
where $2k$ is the dimension of the center of $\frg$.
Proposition~\ref{hol-g7} shows that the converse does not hold in general, since a further reduction to U(1)
may happen for a non-Abelian complex structure. In fact, notice that
by \eqref{equations2-2-g7-bis} the complex structures $J_{\delta}$ on $\frg_7$ are not Abelian.
\end{itemize}
}

\end{remark}

\medskip

In the following result we show that in the semisimple case, i.e. for $\frg=\slC$, the holonomy of the Bismut
connection of the solutions found in Section~\ref{sl2C} always reduces to the subgroup SO(3) inside SU(3).

\begin{proposition}\label{hol-sl2C}
Let us consider a compact quotient of ${\rm SL}(2,\mathbb C)$ endowed with a balanced Hermitian metric
given by~\eqref{2form-sl2C}.
Then, the holonomy of the associated Bismut connection reduces to {\rm SO(3)}.
\end{proposition}

\begin{proof}
From Appendix~\ref{apendice2} for $\varepsilon=\frac12$ and $\rho=0$, we obtain the curvature forms~$(\Omega^+)^i_j$ for
the Bismut connection $\nabla^+$. Now, we proceed as in the proof of Proposition~\ref{hol-g7}, taking into account
\eqref{endo-curv} for the calculation of the curvature endomorphisms $R^+(e_p,e_q)$.
It can be checked directly that the only non-zero endomorphisms $R^+(e_p,e_q)$ are
$$
\begin{array}{lll}
R^+(e_1,e_3) \!\!&\!\!=\!\!&\!\! R^+(e_2,e_4) = -\frac{2}{t^2}(e^{13}+e^{24}),\\[5pt]
R^+(e_1,e_5) \!\!&\!\!=\!\!&\!\! R^+(e_2,e_6) = -\frac{2}{t^2}(e^{15}+e^{26}),\\[5pt]
R^+(e_3,e_5) \!\!&\!\!=\!\!&\!\! R^+(e_4,e_6) = -\frac{2}{t^2}(e^{35}+e^{46}).
\end{array}
$$

Moreover, the non-zero covariant derivatives of the 2-forms $e^{13}+e^{24}$, $e^{15}+e^{26}$ and $e^{35}+e^{46}$ are
the following:
$$
\begin{array}{lll}
& \nabla^+_{e_1}(e^{13}+e^{24}) = -\frac{2}{t}(e^{15}+e^{26}),
\quad\quad
&\nabla^+_{e_3}(e^{13}+e^{24}) = -\frac{2}{t}(e^{35}+e^{46}),\\[8pt]
& \nabla^+_{e_1}(e^{15}+e^{26}) = \frac{2}{t}(e^{13}+e^{24}),
\quad\quad
& \nabla^+_{e_5}(e^{15}+e^{26}) = -\frac{2}{t}(e^{35}+e^{46}),\\[8pt]
& \nabla^+_{e_3}(e^{35}+e^{46}) = \frac{2}{t}(e^{13}+e^{24}),
\quad\quad
& \nabla^+_{e_5}(e^{35}+e^{46}) = \frac{2}{t}(e^{15}+e^{26}).
\end{array}
$$

Therefore, by Ambrose-Singer theorem the Lie algebra of the
holonomy group of the Bismut connection has dimension 3,
and it corresponds to $\mathfrak{s}\mathfrak{o}(3)$
inside $\mathfrak{s}\mathfrak{u}(3)$.
\end{proof}

\medskip

We finish this section by focusing on some cohomological properties of the balanced metrics
studied in the previous sections.
We notice first that for compact quotients of SL(2,$\mathbb{C}$), any invariant form of type (2,2) with respect to the complex
structure \eqref{equations-sl2C-bis} is exact. In fact, it is straightforward to check that
$$
\bigwedge\!^{2,2}(\mathfrak{sl}(2,\mathbb C)^*)
= d \left( \bigwedge\!^{2,1}(\mathfrak{sl}(2,\mathbb C)^*) \right).
$$
Therefore, for any invariant balanced Hermitian metric $F$ we have that the de Rham cohomology class of $F^2$ always vanishes,
i.e. $[F^2]=0$.
More generally, it is announced in \cite{Yachou} that the same holds for any semi-K\"ahler metric on a quotient
of any complex Lie group $G$ which is semisimple.

In contrast, in the nilmanifold and solvmanifold cases we have

\begin{proposition}\label{cohomology}
Let $X$ be a nilmanifold
with underlying Lie algebra $\frh_3$ endowed with the complex structure \eqref{J-h3},
or a solvmanifold with underlying Lie algebra $\frg_7$
endowed with a complex structure \eqref{equations2-2-g7-bis}.
Then,
$
[F^2] \not= 0
$
in $H^4_{{\rm dR}}(X;\mathbb{R})$, for any (not necessarily invariant) balanced metric $F$ on $X$.
\end{proposition}

\begin{proof}
Let $F$ be a generic balanced Hermitian metric on $X$ and suppose that $F^2=d \beta$ for some 3-form $\beta$ on $X$.
Applying the well-known symmetrization process (see for instance \cite{UV2} for details),
one has that there should exist an invariant balanced Hermitian metric $\tilde F$ on $X$ such that
$\tilde F^2 \in d \left( \bigwedge\!^{3}(\frg^*) \right)$. Here $\frg=\frh_3$ when $X$ is a nilmanifold and $\frg=\frg_7$
when $X$ is a solvmanifold.

Now, any invariant balanced metric is given by~\eqref{2form-h3} and~\eqref{2form-g7}, and hence we have:

\vskip.05cm

$\bullet$\ for $\frg=\frh_3$,\ \ $2\, F_t^2=\omega^{12\bar1\bar2}+t^2\omega^{13\bar1\bar3}+t^2\omega^{23\bar2\bar3}$;

\vskip.1cm

$\bullet$\ for $\frg=\frg_7$,\ \ $2\, F_{r,t,u}^2=
(r^4 - |u|^2)\omega_{\delta}^{12\bar1\bar2}+r^2t^2\left(\omega_{\delta}^{13\bar1\bar3}+\omega_{\delta}^{23\bar2\bar3}\right)
+it^2\left(\bar u\, \omega_{\delta}^{23\bar1\bar3}-u\,\omega_{\delta}^{13\bar2\bar3}\right)$.

\vskip.1cm

\noindent From the complex structure equations \eqref{J-h3} and \eqref{equations2-2-g7-bis},
it follows easily that these (2,2)-forms are never exact. In conclusion,
any balanced Hermitian metric $F$ defines a non-zero de Rham class $[F^2]$
in $H^4_{{\rm dR}}(X;\mathbb{R})$.
\end{proof}

For a compact complex manifold $X$, of complex dimension $n$, endowed with a balanced metric $F$,
one can define the map in cohomology
$$
\mathcal{L}\colon H^1_{\rm dR}(X;\mathbb{R})\longrightarrow H^{2n-1}_{\rm dR}(X;\mathbb{R})
$$
given by the cup product with the de Rham cohomology class $[F^{n-1}]\in H^{2n-2}_{\rm dR}(X;\mathbb{R})$.

\begin{proposition}\label{cup-product}
Let $X$ be as in Proposition~\ref{cohomology}. For any invariant balanced metric~$F$ on $X$,
the map $\mathcal{L}\colon H^1_{\rm dR}(X;\mathbb{R})\longrightarrow H^{5}_{\rm dR}(X;\mathbb{R})$
given by the cup product with $[F^{2}]$, is not an isomorphism.
\end{proposition}

\begin{proof}
It is easy to see that the 1-form $e^5$ defines a non-zero class in $H^1_{\rm dR}(X;\mathbb{R})$.
Thus, $\mathcal{L}([e^5])=[2 e^{12345}]$. The class $[e^{12345}]$ is zero in $H^{5}_{\rm dR}(X;\mathbb{R})$ because
in the nilpotent case, i.e. for $\frg=\frh_3$, from \eqref{equations-h3-es} we get that $d e^{1256}= -2t\, e^{12345}$,
and in the solvable case, i.e.
for $\frg=\frg_7$
equations \eqref{equations-g7-es} imply that $d e^{1256}= -\delta \frac{2t}{r^2}\, e^{12345}$.
So, in any case the 5-form $e^{12345}$ is exact and therefore $\mathcal{L}$ is not injective.
\end{proof}

By Proposition~\ref{cup-product}
all our solutions
to the Strominger system (and to the heterotic equations of motion) found in Sections~\ref{h3} and~\ref{g7} share the property that
$\mathcal{L}$ is never an isomorphism.
It is worthy to note that there exist other balanced nilmanifolds solving the Strominger system (but not the heterotic equations of motion)
for which the map $\mathcal{L}$ is an isomorphism.

\bigskip

In the following table we summarize the main results obtained in the paper.

\newpage
\begin{landscape}

\begin{center}
\begin{table}\label{Tabla-resumen}
\renewcommand{\arraystretch}{1.5}
\begin{tabular}{|c|c|c|c|c|c|c|c|c|c|}
\hline
\multirow{2}{*}\textbf{\textbf{Underlying}}&\textbf{Balanced}&\multirow{2}{*}{$\nre$}&\multirow{2}{*}{\textbf{Sign}\! $(\alpha')$}&\multirow{2}{*}{\textbf{Instanton}}&\textbf{Some particular}&\textbf{Heterotic}&\multirow{2}{*}{\textbf{Hol}$(\nabla^+)$}&\textbf{de Rham} \\[-5pt]
\textbf{Lie algebra}&\textbf{metric $F$}& & & &\textbf{connections}&\textbf{eq. motion}& &\textbf{class} $[F^2]$\\
\hline
\multirow{1}{*}{$\frh_3$}&\multirow{2}{*}{\eqref{2form-h3}}&$\rho\geq \varepsilon + \frac12$&$-$&non-flat &
$\nabla^c,\, \nabla^-,\, \nabla^{\textsf{t}\geqslant 1}$& no & \multirow{2}{*}{U(1)}&\multirow{2}{*}{non-trivial} \\
\cline{3-7}
\multirow{1}{*}{(nilpotent)}& &$\rho< \varepsilon + \frac12$&$+$&non-flat &$\nabla^{LC},\, \nabla^+,\, \nabla^{\textsf{t}< 1}$& yes ($\nabla^+$) &&\\
\hline
\multirow{1}{*}{$\frak{sl}(2,\C)$}& \multirow{2}{*}{\eqref{2form-sl2C}} &$\beta(\re)\neq 0$ &\text{sign}\,$\beta(\re)$&flat &$\alpha'>0:\,\,\nabla^{LC}, \nabla^+, \nabla^{\textsf{t}<0}$& yes ($\nabla^+$) &\multirow{2}{*}{SO(3)}&\multirow{2}{*}{trivial}\\
\cline{3-7}
\multirow{1}{*}{(semisimple)}& &$\beta(\re)\neq 8$&\text{sign}\,$(\beta(\re)-8)$&non-flat &$\alpha'>0:\,\,\nabla^{\textsf{t}<-1}$& no &&\\
\hline
\multirow{4}{*}{$\frg_7$}&\eqref{2form-g7}&$\rho\geq \varepsilon + \frac12$&$-$&non-flat &$\nabla^c,\, \nabla^-,\, \nabla^{\textsf{t}\geqslant 1}$& no &\multirow{2}{*}{U(1)}&
\multirow{5}{*}{non-trivial}\\
\cline{3-7}
&$u=0$&$\rho< \varepsilon + \frac12$&$+$&non-flat &$\nabla^{LC},\, \nabla^+,\, \nabla^{\textsf{t}< 1}$& yes ($\nabla^+$) &&\\
\cline{2-8}
\multirow{2}{*}{(solvable)}&\multirow{2}{*}{\eqref{2form-g7}}&$P_1$, $Q_1$&$+$&flat && no &\multirow{3}{*}{SU(3)}&\\
\cline{3-7}
&\multirow{2}{*}{$u\neq0$}&$\Delta \backslash \Delta_+$&$-$&non-flat &$\nabla^{\textsf{t}\in I}$& no &&\\
\cline{3-7}
& &$\Delta_+$&$+$&non-flat &$\nabla^+,\,\,\nabla^c,\,\, \nabla^{\textsf{t}\in I_+}$& no &&\\
\hline
\end{tabular}
\bigskip
\caption{Invariant solutions to the Strominger system with respect to the connections $\nre$
in the anomaly cancellation condition.}
\end{table}
\end{center}

\vspace{1cm}

\hspace{.1cm} $\bullet$ The rows starting with ``$\frh_3$'' collect the main results for a nilmanifold with underlying Lie algebra $\frh_3$
obtained in Section~\ref{h3}

\hspace{.3cm} (see Theorem~\ref{solutions-h3}). For the description of the connections $\nre$ see Figures~\ref{conexiones} and~\ref{conexiones-3}.

\vspace{.2cm}

\hspace{.1cm} $\bullet$ The rows ``$\slC$'' correspond to the solutions on a compact quotient of ${\rm SL}(2,\mathbb C)$ found in
Section~\ref{sl2C} (see Theorem~\ref{solutions-sl2C}).

\hspace{.3cm} The value of $\beta(\re)$ is given in \eqref{beta-bis}.

\vspace{.2cm}

\hspace{.1cm} $\bullet$ The rows starting with ``$\frg_7$'' summarize the main results for the solvmanifolds with underlying Lie algebra~$\frg_7$
obtained in Section~\ref{g7},

\hspace{.3cm} that is, Theorem~\ref{solutions-g7-u=0} for balanced metrics \eqref{2form-g7} with $u=0$,
and Theorem~\ref{solutions-g7-unot0} for metrics with $u\not=0$
(see also Corollaries~\ref{solutions-g7-unot0-cor1} and~\ref{solutions-g7-unot0-cor2}).


\hspace{.3cm} The points $P_1$ and $Q_1$, and the regions
$\Delta$ and $\Delta_+ \subset \Delta$ are described in Section~\ref{g7} (see Figure~\ref{figura}).
Here, $I$ and $I_+ \subset I$ are given by

\hspace{.3cm} $I=\Delta \cap l=(-5-4\sqrt{2},3+2\sqrt{2})$ and
$I_+=\Delta_+ \cap l=(-5-4\sqrt{2},1-\sqrt{2}) \cup (\frac{5-\sqrt{17}}{2},1+\sqrt{2})$, where
$l$ denotes the line corresponding

\hspace{.3cm} to the Hermitian connections $\nabla^{\textsf{t}}$.

\vspace{.2cm}

\hspace{.1cm} $\bullet$ The last two columns correspond to the results given in Section~\ref{holonomy-cohomology}.

\end{landscape}

\section{Conclusions}\label{Conclusions}

\noindent
We construct many new solutions to the Strominger system in six dimensions
with respect to a 2-parameter family of metric connections $\nre$ in the anomaly cancellation equation.
All the solutions constructed in this paper are non-K\"ahler and with non vanishing flux.
More concretely, the solutions are all invariant solutions
living on three different compact non-K\"ahler homogeneous spaces, which are obtained as the quotient, by a lattice of maximal rank,
of a nilpotent Lie group (the nilmanifold $\frh_3$),
the semisimple group SL(2,$\mathbb{C}$) and a solvable Lie group (the solvmanifold $\frg_7$).
As the solutions are invariant, the underlying Hermitian metrics are balanced and the dilaton is constant.

The solutions are obtained after a careful analysis of the first Pontrjagin form of the connections $\nre$,
and as a consequence many new solutions to the Strominger system with non-flat instanton and string tension $\alpha'$
of different signs are obtained.
Since the ansatz $\nre$ is a natural
extension of the canonical 1-parameter family of Hermitian connections found by Gauduchon \cite{Gau},
in particular we provide solutions with respect to the Chern connection $\nabla^c$ or the (Strominger-)Bismut connection $\nabla^+$
in the anomaly cancellation equation, with non-flat instanton and $\alpha'>0$.
Moreover, the Levi-Civita connection $\nabla^{LC}$ and the connection $\nabla^-$ also belong to the ansatz $\nre$,
and solutions are given in these cases, too. All these connections were proposed
(see for instance~\cite{Str,GP,Fu-Yau,CCDLMZ,FIUV09} and the references therein) for the anomaly cancellation equation
in the Strominger system.

For the nilmanifold $\frak h_3$ and for the quotient of the semisimple group SL(2,$\mathbb{C}$),
we extend the study of invariant solutions developed in \cite{FIUV09} and \cite{FeiYau}, respectively,
to the family of connections $\nre$.
The solvmanifold $\frg_7$ was found in \cite{FOU} as a new compact complex (non-K\"ahler) homogeneous space with
holomorphically trivial canonical bundle endowed with balanced Hermitian metrics.
We show here that this new solvmanifold provides many invariant solutions to the Strominger system with respect to the
connections $\nre$ in the anomaly cancellation equation, including the Bismut and the Chern connections,
with non-flat instanton and positive string tension $\alpha'$.

Solutions to the heterotic equations of motion are also found in this paper.
We make use of a result by Ivanov \cite{Iv} asserting that a solution of the Strominger system satisfies the heterotic equations
of motion if and only if the connection
$\nabla$ in the anomaly cancellation equation is an instanton.
For the previous homogeneous spaces, we prove that a connection $\nabla=\nre$ in our ansatz is a non-flat instanton if and only if
$\nabla=\nabla^+$, i.e. it is the Bismut connection. This allows us to give explicit solutions
to the heterotic equations of motion, with respect to the Bismut connection and with positive string tension $\alpha'$,
on the three compact non-K\"ahler homogeneous spaces, i.e. on the nilmanifold $\frh_3$ (see also \cite{FIUV09}), on
the quotient of the semisimple group SL(2,$\mathbb{C}$) and on the solvmanifold $\frg_7$.
In the SL(2,$\mathbb{C}$) case, this provides an affirmative answer to a question posed
by Andreas and Garc\'{\i}a-Fern\'andez in~\cite{AGarcia}.

To our knowledge, these are the only known invariant solutions
to the heterotic equations of motion in six dimensions.
We conjecture that if a compact non-K\"ahler homogeneous space $M=G/\Gamma$
admits an invariant solution to the heterotic equations of motion with $\alpha'>0$ and
with respect to some connection $\nabla$ in the ansatz $\nre$,
then
$\nabla$ is the Bismut connection $\nabla^+$ and
$M$ is the nilmanifold $\frh_3$,
a quotient of the semisimple group SL(2,$\mathbb{C}$) or the solvmanifold $\frg_7$.

Finally, we observe that
the results in \cite{FIUV2014} suggest that the solutions constructed in this paper could serve as a source for the construction of
new smooth supersymmetric solutions to
the heterotic equations of motion up to first order of $\alpha'$ with non vanishing flux,
non-flat instanton and non-constant dilaton.
Indeed, the source of the construction in \cite{FIUV2014} was
the invariant solutions to the Strominger system with constant
dilaton previously found on nilmanifolds.

\newpage


\section{Appendix}\label{apendice}

\noindent
In this appendix we present the curvature 2-forms $(\cre)^i_j$, of any connection $\nre$ in the
family introduced in Section~\ref{family-connections},
on the compact quotient manifolds corresponding to the cases $\frg=\frh_3, \slC$ and $\frg_7$.
The curvature forms $(\cre)^i_j$ are derived by a direct calculation using \eqref{curvature} and \eqref{connection-1-forms}
in an appropriate adapted basis of 1-forms on $\frg$.

\subsection{Curvature forms $(\cre)^i_j$ for the nilpotent case $\frg=\frh_3$}\label{apendice1}

The forms $(\cre)^i_j$ are derived directly from \eqref{curvature}--\eqref{connection-1-forms}
and the structure equations \eqref{equations-h3-es} with respect to the basis of 1-forms $\{e^1,\ldots,e^6\}$,
which is adapted to the SU(3)-structure $(J^-,F_t,\Psi_t)$, $t \not=0$,
given by~\eqref{adapted-basis-h3}-\eqref{Psi-h3}. We have:
\begin{equation*}
\begin{array}[t]{rl}
(\cre)^1_2=\!\! & \!\!
- (3+4\,\varepsilon^2-4\,\rho+4\,\rho^2) t^2\, e^{12}
+ 2(1+2\,\varepsilon-2\,\rho) t^2\, e^{34}, \\[8pt]
(\cre)^1_3=\!\! & \!\!
4\,\rho^2 t^2 \, e^{13}
+ (1-2\,\varepsilon)^2 t^2 \, e^{24}, \\[8pt]
(\cre)^1_4=\!\! & \!\!
4\,\rho^2 t^2 \, e^{14}
- (1-2\,\varepsilon)^2 t^2 \, e^{23}, \\[8pt]
(\cre)^1_5=\!\! & \!\!
2(1+2\,\varepsilon-2\,\rho) \rho\, t^2\, e^{26}, \\[8pt]
(\cre)^1_6=\!\! & \!\!
(1-2\,\varepsilon)(1+2\,\varepsilon-2\,\rho) t^2\, e^{16}, \\[8pt]
(\cre)^2_3=\!\! & \!\!
-(1-2\,\varepsilon)^2 t^2 \, e^{14}
+ 4\,\rho^2 t^2 \, e^{23}, \\[8pt]
(\cre)^2_4=\!\! & \!\!
(1-2\,\varepsilon)^2 t^2 \, e^{13} + 4\,\rho^2 t^2 \, e^{24}, \\[8pt]
(\cre)^2_5=\!\! & \!\!
-2(1+2\,\varepsilon-2\,\rho) \rho\, t^2\, e^{16}, \\[8pt]
(\cre)^2_6=\!\! & \!\!
(1-2\,\varepsilon)(1+2\,\varepsilon-2\,\rho) t^2\, e^{26}, \\[8pt]
(\cre)^3_4=\!\! & \!\!
2(1+2\,\varepsilon-2\,\rho) t^2\, e^{12}
- (3+4\,\varepsilon^2-4\,\rho+4\,\rho^2) t^2\, e^{34}, \\[8pt]
(\cre)^3_5=\!\! & \!\!
2(1+2\,\varepsilon-2\,\rho) \rho\, t^2\, e^{46}, \\[8pt]
(\cre)^3_6=\!\! & \!\!
(1-2\,\varepsilon)(1+2\,\varepsilon-2\,\rho) t^2\, e^{36}, \\[8pt]
(\cre)^4_5=\!\! & \!\!
-2(1+2\,\varepsilon-2\,\rho) \rho\, t^2\, e^{36}, \\[8pt]
(\cre)^4_6=\!\! & \!\!
(1-2\,\varepsilon)(1+2\,\varepsilon-2\,\rho) t^2\, e^{46}, \\[8pt]
(\cre)^5_6=\!\! & \!\!
4(1-2\,\varepsilon)\rho t^2\, e^{12}
+4(1-2\,\varepsilon)\rho t^2\, e^{34}. \\[8pt]
\end{array}
\end{equation*}

It is worth noticing that the family of Hermitian connections $\nabla^{\textsf{t}}$ (i.e. $\rho=\frac12-\varepsilon$)
satisfies the relations
\begin{equation}\label{simetrias-curv}
\begin{array}{rl}
\!\! & \!\!  (\Omega^{\textsf{t}})^1_2+(\Omega^{\textsf{t}})^3_4+(\Omega^{\textsf{t}})^5_6=0,\quad
(\Omega^{\textsf{t}})^2_3=-(\Omega^{\textsf{t}})^1_4,\quad
(\Omega^{\textsf{t}})^2_4=(\Omega^{\textsf{t}})^1_3,  \\[8pt]
\!\! & \!\!  (\Omega^{\textsf{t}})^2_5=-(\Omega^{\textsf{t}})^1_6,\quad
(\Omega^{\textsf{t}})^2_6=(\Omega^{\textsf{t}})^1_5,\quad
(\Omega^{\textsf{t}})^4_5=-(\Omega^{\textsf{t}})^3_6,\quad
(\Omega^{\textsf{t}})^4_6=(\Omega^{\textsf{t}})^3_5.
\end{array}
\end{equation}

\subsection{Curvature forms $(\cre)^i_j$ for the semisimple case $\frg=\slC$}\label{apendice2}

The forms $(\cre)^i_j$ are derived directly from \eqref{curvature}--\eqref{connection-1-forms}
and the structure equations \eqref{equations-sl2C-es-bis} with respect to the basis of 1-forms $\{e^1,\ldots,e^6\}$,
which is adapted to the SU(3)-structure $(J,F_t,\Psi_t)$, $t \not=0$,
given by~\eqref{adapted-basis-sl2C}-\eqref{Psi-sl2C}:

\begin{equation*}
\begin{array}[t]{rl}
(\cre)^1_2=\!\! & \!\!
- \frac{(1-2\,\varepsilon-2\,\rho)^2}{2 t^2} (e^{34}+e^{56}), \\[8pt]
(\cre)^1_3=\!\! & \!\!
\frac{1}{4 t^2} \left[ (1+6\,\varepsilon-2\,\rho)(1-6\,\varepsilon+2\,\rho) e^{13}
- (3+8\,\varepsilon+4\,\varepsilon^2-8\,\rho+8\,\varepsilon\rho+4\,\rho^2) e^{24} \right], \\[8pt]
(\cre)^1_4=\!\! & \!\!
\frac{1-2\,\varepsilon-2\,\rho}{4 t^2} \left[ (1+2\,\varepsilon-6\,\rho) e^{14}
- (1-6\,\varepsilon+2\,\rho) e^{23} \right], \\[8pt]
(\cre)^1_5=\!\! & \!\!
\frac{1}{4 t^2} \left[ (1+6\,\varepsilon-2\,\rho)(1-6\,\varepsilon+2\,\rho) e^{15}
- (3+8\,\varepsilon+4\,\varepsilon^2-8\,\rho+8\,\varepsilon\rho+4\,\rho^2) e^{26} \right], \\[8pt]
(\cre)^1_6=\!\! & \!\!
\frac{1-2\,\varepsilon-2\,\rho}{4 t^2} \left[ (1+2\,\varepsilon-6\,\rho) e^{16}
- (1-6\,\varepsilon+2\,\rho) e^{25} \right], \\[8pt]
(\cre)^2_3=\!\! & \!\!
-\frac{1-2\,\varepsilon-2\,\rho}{4 t^2} \left[ (1-6\,\varepsilon+2\,\rho) e^{14}
- (1+2\,\varepsilon-6\,\rho) e^{23} \right],
\end{array}
\end{equation*}

\begin{equation*}
\begin{array}[t]{rl}
(\cre)^2_4=\!\! & \!\!
-\frac{1}{4 t^2} \left[ (3+2\,\varepsilon-6\,\rho)(1+2\,\varepsilon-6\,\rho) e^{13}
+ (7+4\,\varepsilon^2-16\,\rho+8\,\varepsilon\rho+4\,\rho^2) e^{24} \right], \\[8pt]
(\cre)^2_5=\!\! & \!\!
-\frac{1-2\,\varepsilon-2\,\rho}{4 t^2} \left[ (1-6\,\varepsilon+2\,\rho) e^{16}
- (1+2\,\varepsilon-6\,\rho) e^{25} \right], \\[8pt]
(\cre)^2_6=\!\! & \!\!
-\frac{1}{4 t^2} \left[ (3+2\,\varepsilon-6\,\rho)(1+2\,\varepsilon-6\,\rho) e^{15}
+ (7+4\,\varepsilon^2-16\,\rho+8\,\varepsilon\rho+4\,\rho^2) e^{26} \right],\\[8pt]
(\cre)^3_4=\!\! & \!\!
- \frac{(1-2\,\varepsilon-2\,\rho)^2}{2 t^2} (e^{12}+e^{56}), \\[8pt]
(\cre)^3_5=\!\! & \!\!
\frac{1}{4 t^2} \left[ (1+6\,\varepsilon-2\,\rho)(1-6\,\varepsilon+2\,\rho) e^{35}
- (3+8\,\varepsilon+4\,\varepsilon^2-8\,\rho+8\,\varepsilon\rho+4\,\rho^2) e^{46} \right], \\[8pt]
(\cre)^3_6=\!\! & \!\!
\frac{1-2\,\varepsilon-2\,\rho}{4 t^2} \left[ (1+2\,\varepsilon-6\,\rho) e^{36}
- (1-6\,\varepsilon+2\,\rho) e^{45} \right],\\[8pt]
(\cre)^4_5=\!\! & \!\!
-\frac{1-2\,\varepsilon-2\,\rho}{4 t^2} \left[ (1-6\,\varepsilon+2\,\rho) e^{36}
- (1+2\,\varepsilon-6\,\rho) e^{45} \right], \\[8pt]
(\cre)^4_6=\!\! & \!\!
-\frac{1}{4 t^2} \left[ (3+2\,\varepsilon-6\,\rho)(1+2\,\varepsilon-6\,\rho) e^{35}
+ (7+4\,\varepsilon^2-16\,\rho+8\,\varepsilon\rho+4\,\rho^2) e^{46} \right], \\[8pt]
(\cre)^5_6=\!\! & \!\!
- \frac{(1-2\,\varepsilon-2\,\rho)^2}{2 t^2} (e^{12}+e^{34}).
\end{array}
\end{equation*}

Notice that the family of Hermitian connections $\nabla^{\textsf{t}}$ (i.e. $\rho=\frac12-\varepsilon$)
satisfies the relations~\eqref{simetrias-curv}.

\subsection{Curvature forms $(\cre)^i_j$ for the solvable case $\frg=\frg_7$}\label{apendice3}

The forms $(\cre)^i_j$ are derived by a long, but direct, calculation using
the structure equations \eqref{equations-g7-es} in terms of the adapted basis $\{e^1,\ldots,e^6\}$,
and taking into account \eqref{curvature} and \eqref{connection-1-forms}. We recall that $\delta=\pm 1$
depends only on the complex structure, and $r,t,u_1=\Real u, u_2=\Imag u$ are (real) coefficients
that describe the balanced $J_{\delta}$-Hermitian metrics on $\frg_7$.

We found the following relations:

\begin{equation*}
\begin{array}[t]{rl}
(\cre)^1_2+(\cre)^3_4+(\cre)^5_6=\!\! & \!\!
- (1-2\,\varepsilon-2\,\rho)^2\, \frac{t^4+4|u|^2}{t^2(r^4-|u|^2)} \, (e^{12}+e^{34}), \\[8pt]
(\cre)^1_3-(\cre)^2_4=\!\! & \!\!
- (1-2\,\varepsilon-2\,\rho)\, \frac{(1-2\,\varepsilon+2\,\rho)t^4-4(1+2\,\varepsilon-2\,\rho)|u|^2}{t^2(r^4-|u|^2)} \, (e^{13}-e^{24}), \\[8pt]
(\cre)^1_4+(\cre)^2_3=\!\! & \!\!
- (1-2\,\varepsilon-2\,\rho)\, \frac{(1-2\,\varepsilon+2\,\rho)t^4-4(1+2\,\varepsilon-2\,\rho)|u|^2}{t^2(r^4-|u|^2)} \, (e^{14}+e^{23}),
\end{array}
\end{equation*}

\begin{equation*}
\begin{array}[t]{rl}
(\cre)^1_5-(\cre)^2_6=\!\! & \!\!
\frac{1-2\,\varepsilon-2\,\rho}{\sqrt{r^4-|u|^2}}\, \Big[\,
\frac{-12\, |u|^2}{t^2\sqrt{r^4-|u|^2}} \, e^{15}
+ \frac{2\,\delta\,|u|^2}{r^2\sqrt{r^4-|u|^2}}\, (e^{16}-e^{25})
- \frac{(1+2\,\varepsilon-2\,\rho)\, t^4 + 8 (\varepsilon -\rho)\,|u|^2}{t^2\sqrt{r^4-|u|^2}}\, e^{26} \\[8pt]
\!\! & \!\!
- \frac{2\,\delta\, u_2}{r^2}\, (e^{35} + e^{46}) - \frac{2\,\delta\, u_1}{r^2}\, (e^{36} - e^{45})
+ \frac{8}{t^2} (u_1\, e^{35} + u_2\,e^{45}) \,\Big]\, ,
\\[10pt]
(\cre)^1_6+(\cre)^2_5=\!\! & \!\!
\frac{1-2\,\varepsilon-2\,\rho}{\sqrt{r^4-|u|^2}}\, \Big[\,
\frac{2\, \delta\, |u|^2}{r^2\sqrt{r^4-|u|^2}} \, (e^{15}+e^{26})
+ \frac{(1+2\,\varepsilon-2\,\rho)\, t^4 + 8 (\varepsilon -\rho)\,|u|^2}{t^2\sqrt{r^4-|u|^2}}\, e^{16}
-\frac{12\, |u|^2}{t^2\sqrt{r^4-|u|^2}} \, e^{25}  \\[8pt]
\!\! & \!\! - \frac{2\,\delta\, u_1}{r^2}\, (e^{35} + e^{46}) + \frac{2\,\delta\, u_2}{r^2}\, (e^{36} - e^{45})
- \frac{8}{t^2} (u_2\, e^{35} - u_1\,e^{45}) \,\Big]\, ,\\[8pt]
(\cre)^3_5-(\cre)^4_6=\!\! & \!\!
\frac{1-2\,\varepsilon-2\,\rho}{\sqrt{r^4-|u|^2}}\, \Big[\,
\frac{2\,\delta\, u_2}{r^2}\, (e^{15} + e^{26}) - \frac{2\,\delta\, u_1}{r^2}\, (e^{16} - e^{25})
+ \frac{8}{t^2} (u_1\, e^{15} - u_2\,e^{25}) \\[8pt]
\!\! & \!\!
+ \frac{4\, |u|^2}{t^2\sqrt{r^4-|u|^2}} \, e^{35}
- \frac{2\, \delta\, |u|^2}{r^2\sqrt{r^4-|u|^2}} (e^{36}-e^{45})
- \frac{(1+2\,\varepsilon-2\,\rho)\, t^4 + 8 (\varepsilon -\rho)\,|u|^2}{t^2\sqrt{r^4-|u|^2}}\, e^{46} \,\Big]\, , \\[10pt]
(\cre)^3_6+(\cre)^4_5=\!\! & \!\!
\frac{1-2\,\varepsilon-2\,\rho}{\sqrt{r^4-|u|^2}}\, \Big[\,
-\frac{2\,\delta\, u_1}{r^2}\, (e^{15} + e^{26}) - \frac{2\,\delta\, u_2}{r^2}\, (e^{16} - e^{25})
+ \frac{8}{t^2} (u_2\, e^{15} + u_1\,e^{25}) \\[8pt]
\!\! & \!\! - \frac{2\, \delta\, |u|^2}{r^2\sqrt{r^4-|u|^2}} (e^{35}+e^{46})
+ \frac{(1+2\,\varepsilon-2\,\rho)\, t^4 + 8 (\varepsilon -\rho)\,|u|^2}{t^2\sqrt{r^4-|u|^2}}\, e^{36}
+ \frac{4\, |u|^2}{t^2\sqrt{r^4-|u|^2}} \, e^{45} \,\Big]\, .
\end{array}
\end{equation*}

Notice that, as in the previous cases~\ref{apendice1} and~\ref{apendice2},
the Hermitian connections $\nabla^{\textsf{t}}$ (i.e. $\rho=\frac12-\varepsilon$)
satisfy the relations~\eqref{simetrias-curv}.

Taking into account the general relations above, it suffices to describe the 8 curvature 2-forms
$(\cre)^1_2$, $(\cre)^1_3$, $(\cre)^1_4$, $(\cre)^1_5$, $(\cre)^1_6$,
$(\cre)^3_4$, $(\cre)^3_5$ and $(\cre)^3_6$.
These curvature forms are:

\begin{equation*}
\begin{array}[t]{rl}
(\cre)^1_2=\!\! & \!\!
-\frac{t^2 B}{r^4} e^{12} +\frac{2r^4t^4 (1+2\varepsilon-2\rho)  - |u|^2 \left( t^4B + r^4D\right)}{r^4t^2(r^4-|u^2|)} e^{34}+ \frac{16(\varepsilon - \rho) |u|^2}{t^2(r^4-|u|^2)} e^{56} \\[7pt]
&-\frac{1}{r^4\sqrt{r^4-|u|^2}} \begin{pmatrix}e^{14}-e^{23}&e^{13}+e^{24}\end{pmatrix} \begin{pmatrix}u_2&-u_1\\ u_1&u_2\end{pmatrix} \begin{pmatrix}\delta r^2A\\ - t^2B\end{pmatrix},
\\[20pt]
(\cre)^1_3=\!\! & \!\!
 \frac{-1}{r^4\sqrt{r^4-|u|^2}}\left(\delta r^2 u_1C- t^2u_2B\right)\left(e^{12}-e^{34}\right)+\frac{16(\varepsilon-\rho)u_2}{t^2\sqrt{r^4-|u|^2}}e^{56} \\[7pt]
&+\frac{u_2}{r^4(r^4-|u|^2)}\left(16 \delta  \left(\varepsilon^2 - \rho(1-\rho)\right) r^2 u_1- t^2u_2B\right)(e^{13}+e^{24})\\[7pt]
&-\frac{u_1}{r^4(r^4-|u|^2)}\left(\delta r^2 u_1C- t^2u_2B\right)\left(e^{14}-e^{23}\right)+\frac{u_2}{r^2t^2(r^4-|u|^2)}\left(\delta t^2 u_2 A -r^2 u_1 D\right)(e^{14}-e^{23})\\[7pt]
&+\frac{4}{t^2(r^4-|u|^2)}\left(\rho^2t^4-4\varepsilon^2u_1^2+u_2^2(1-2\rho)^2\right)e^{13}\\[7pt]
&+\frac{1}{t^2(r^4-|u|^2)}\left(t^4(1-2\varepsilon)^2-4u_1^2(1-2\rho)^2+16 u_2^2 \varepsilon^2\right)e^{24},\\[20pt]

(\cre)^1_4=\!\! & \!\!
\frac{-1}{r^4\sqrt{r^4-|u|^2}} \left(t^2 u_1 B+ \delta r^2  u_2 C\right) (e^{12} - e^{34}) - \frac{16(\varepsilon-\rho) u_1}{t^2\sqrt{r^4-|u|^2}} e^{56}-
\frac{\delta}{r^2(r^4-|u|^2)} \left(u_1^2 A - u_2^2 C\right)(e^{13} + e^{24}) \\[7pt]
& +\frac{u_1u_2}{r^4t^2(r^4-|u|^2)} \left(t^4 B  - r^4 D  \right)(e^{13} + e^{24}) - \frac{u_1}{r^4(r^4-|u|^2)}\left(t^2  u_1 B +
16 \delta(\varepsilon^2 - \rho(1 - \rho)) r^2  u_2 \right) (e^{14}-e^{23})\\[7pt]
&+\frac{4 }{t^2 (r^4-|u|^2)}\left(\rho^2 t^4 +  (1-2\rho)^2 u_1^2- 4\varepsilon^2 u_2^2\right) e^{14}-
\frac{1}{t^2 (r^4-|u|^2)}\left(t^4 (1-2\varepsilon)^2 + 16\varepsilon^2 u_1^2  - 4 (1-2\rho)^2 u_2^2  \right) e^{23},\\[20pt]
(\cre)^1_5=\!\! & \!\! \frac{-12 (1-2 \rho)|u|^2}{t^2 (r^4-|u|^2)}e^{15}+\frac{4\delta\rho |u|^2}{r^2 (r^4-|u|^2)}e^{25}+\frac{2\left(\rho t^4(1+2\varepsilon - 2\rho) - 4(1-2\rho)(\varepsilon-\rho)|u|^2\right)}{t^2(r^4-|u|^2)}e^{26}\\[7pt]
&-\frac{2\delta(1+2\varepsilon-4\rho)}{r^2\sqrt{r^4-|u|^2}}(u_1 e^{36} + u_2 e^{46} - \frac{|u|^2}{\sqrt{r^4-|u|^2}} e^{16})+\frac{4}{r^2t^2\sqrt{r^4-|u|^2}} \begin{pmatrix}e^{35}&e^{45}\end{pmatrix} \begin{pmatrix}u_2&u_1\\ -u_1&u_2\end{pmatrix} \begin{pmatrix}\delta \rho t^2\\ 2 r^2 (1-2\rho)\end{pmatrix}, \\[20pt]
(\cre)^1_6=\!\! & \frac{2\delta (1-2 \varepsilon)|u|^2}{r^2 (r^4-|u|^2)}e^{15}- \frac{4\varepsilon(\varepsilon-\rho)(t^4+4|u|^2)-t^4(1-2\rho)}{t^2(r^4-|u|^2)} e^{16}+
\frac{24 \varepsilon|u|^2}{t^2 (r^4-|u|^2)}e^{25} \\
&+\frac{4\delta(2\varepsilon-\rho)}{r^2\sqrt{r^4-|u|^2}}\left(\frac{-|u|^2}{\sqrt{r^4-|u|^2}} e^{26} + u_1e^{46}-u_2e^{36}\right)+
\frac{2}{r^2 t^2 \sqrt{r^4-|u|^2}} \begin{pmatrix}e^{35}&e^{45}\end{pmatrix} \begin{pmatrix}u_2&-u_1\\ -u_1&-u_2\end{pmatrix} \begin{pmatrix}8\,\varepsilon\,r^2\\ \delta t^2 (1-2\varepsilon)\end{pmatrix},
\\[20pt]

(\cre)^3_4=\!\! & \!\!
\frac{2r^4t^4 (1+2\varepsilon-2\rho)  - |u|^2 \left( t^4B + r^4D\right)}{r^4t^2(r^4-|u^2|)} e^{12}-\frac{t^2B}{r^4} e^{34} - \frac{16(\varepsilon - \rho) |u|^2}{t^2(r^4-|u|^2)} e^{56} \\[7pt]
&+\frac{1}{r^4\sqrt{r^4-|u|^2}} \begin{pmatrix}e^{14}-e^{23}&e^{13}+e^{24}\end{pmatrix} \begin{pmatrix}u_2&-u_1\\ u_1&u_2\end{pmatrix} \begin{pmatrix}\delta r^2A\\ - t^2B\end{pmatrix},
\\[20pt]

(\cre)^3_5=\!\! & \!\!
\frac{2\delta (1+2\varepsilon - 4\rho)}{r^2\sqrt{r^4-|u|^2}}\left(-u_1 e^{16} + u_2 e^{26} - \frac{|u|^2}{\sqrt{r^4-|u|^2}} e^{36}\right) + \frac{4|u|^2}{r^4-|u|^2} \left(\frac{1-2\rho}{t^2} e^{35} - \frac{\delta \rho}{r^2} e^{45}\right) \\[7pt]
&+\frac{2 \left(\rho t^4(1+2\varepsilon - 2\rho) -4(1-2\rho)(\varepsilon-\rho)|u|^2\right) }{t^2(r^4-|u|^2)} e^{46}- \frac{4}{r^2 t^2 \sqrt{r^4-|u|^2}} \begin{pmatrix}e^{15}&e^{25}\end{pmatrix} \begin{pmatrix}u_2&-u_1\\ u_1&u_2\end{pmatrix} \begin{pmatrix}\delta \rho\,t^2\\ 2 r^2 (1-2\rho)\end{pmatrix},\\[20pt]

(\cre)^3_6=\!\! & \!\!
\frac{4\delta(2\varepsilon-\rho)}{r^2\sqrt{r^4-|u|^2}} (u_2 e^{16} + u_1 e^{26}) - \frac{2\delta (1-2\varepsilon) |u|^2}{r^2(r^4-|u|^2)} e^{35} -
\frac{4|u|^2}{r^4-|u|^2} \left(\frac{2\varepsilon}{t^2} e^{45} + \frac{(2\varepsilon - \rho)\delta}{r^2} e^{46}\right)\\[7pt]
&-\frac{4\varepsilon(\varepsilon-\rho)(t^4+4|u|^2)-t^4(1-2\rho)}{t^2(r^4-|u|^2)} e^{36}-
\frac{2}{r^2 t^2 \sqrt{r^4-|u|^2}} \begin{pmatrix}e^{15}&e^{25}\end{pmatrix} \begin{pmatrix}u_2&-u_1\\ u_1&u_2\end{pmatrix} \begin{pmatrix}8\,\varepsilon\,r^2\\ -\delta t^2 (1-2\varepsilon)\end{pmatrix},

\end{array}
\end{equation*}
where

\vskip.2cm

$A=A(\re) = 4\left( 2\varepsilon^2+ 2\rho^2-\varepsilon - \rho\right)$,\ \qquad $B = B(\re) = 4 (\varepsilon^2  + \rho^2- \rho+3/4)$,

\vskip.2cm

$C=C(\re) = 4  ( 2\varepsilon^2 + 2\rho^2+\varepsilon-3\rho)$,\qquad $D=D(\re)= 4\left(4\varepsilon^2+4\rho^2-4\rho+1\right)$.

\section*{Acknowledgments}
\noindent This work has been partially supported by the projects MINECO (Spain) MTM2014-58616-P,
and Gobi\-er\-no de Arag\'on/Fondo Social Europeo, grupo consolidado E15-Geometr\'{\i}a.
We wish to thank Mario Garc\'{\i}a-Fern\'andez and Stefan Ivanov for useful comments and conversations.
We are grateful to the referee for helpful comments and suggestions.



\vspace{-0.25cm}

\begin{thebibliography}{33}

\bibitem{Str} A. Strominger, \emph{Superstrings with torsion},
Nuclear Phys. {\bf B} \textbf{274} (1986), 253--284.

\bibitem {Hull} C. Hull,
\emph{Superstring compactifications with torsion and space-time supersymmetry}, In
Turin 1985 Proceedings ``Superunification and Extra Dimensions'' (1986) 347--375.

\bibitem{BBDGE} K. Becker, M. Becker, K. Dasgupta, P.S. Green, E. Sharpe,
\emph{Compactifications of heterotic strings on non-K\"ahler complex manifolds: II},
Nuclear Phys. {\bf B} {\bf 678} (2004), 19--100.

\bibitem{MGarcia} M. Garc\'{\i}a-Fern\'andez, \emph{Lectures on the Strominger system},
arXiv:1609.02615 [math.DG].

\bibitem{Li-Yau} J. Li, S.-T. Yau, \emph{The existence of supersymmetric
string theory with torsion}, J. Diff. Geom. \textbf{70}
(2005), 143--181.

\bibitem{A-Gar} B. Andreas, M. Garc\'{\i}a-Fern\'andez,
\emph{Solutions of the Strominger system via stable bundles on Calabi-Yau threefolds},
Commun. Math. Phys. {\bf 315} (2012), 153--168.

\bibitem{GP} E. Goldstein, S. Prokushkin, \emph{Geometric model for complex
non-K\"ahler manifolds with SU(3) structure}, Commun. Math. Phys.
{\bf 251} (2004), 65--78.

\bibitem{Fu-Yau} J.-X. Fu, S.-T. Yau, \emph{The theory of superstring with
flux on non-K\"ahler manifolds and the complex Monge-Amp\`ere equation},
J. Diff. Geom. \textbf{78} (2008), 369--428.

\bibitem{CCDLMZ} G.L. Cardoso, G. Curio, G. Dall'Agata, D. Lust, P. Manousselis, G. Zoupanos,
\emph{Non-K\"aehler string back-grounds and their five torsion
classes}, Nuclear Phys. {\bf B} {\bf 652} (2003), 5--34.

\bibitem{FIUV09} M. Fern\'andez, S. Ivanov, L. Ugarte, R. Villacampa, \emph{Non-K\"ahler
heterotic string compactifications with non-zero fluxes and constant dilaton},
Commun. Math. Phys. {\bf 288} (2009), 677--697.

\bibitem{FeiYau} T. Fei, S.-T. Yau,
\emph{Invariant solutions to the Strominger system
on complex Lie groups and their quotients}, Commun. Math. Phys. {\bf 338} (2015), 1183--1195.

\bibitem{Gau} P. Gauduchon, \emph{Hermitian connections and Dirac operators},
Boll. Un. Mat. Ital. B (7) {\bf 11} (1997), 257--288.

\bibitem{UV2} L. Ugarte, R. Villacampa, \emph{Balanced Hermitian geometry on 6-dimensional nilmanifolds},
Forum Math. {\bf 27} (2015), 1025--1070.

\bibitem{UV1} L. Ugarte, R. Villacampa, \emph{Non-nilpotent complex geometry of nilmanifolds and heterotic
supersymmetry}, Asian J. Math. {\bf 18} (2014), 229--246.

\bibitem{FIUV2014} M. Fern\'andez, S. Ivanov, L. Ugarte, D. Vassilev,
\emph{Non-Kaehler heterotic string solutions with non-zero fluxes and non-constant dilaton},
J. High Energy Physics JHEP {\bf 06} (2014) 073.

\bibitem{AGarcia} B. Andreas, M. Garc\'{\i}a-Fern\'andez, \emph{Note on solutions of the Strominger system from unitary
representations of cocompact lattices of SL(2,$\mathbb{C}$)},
Commun. Math. Phys. {\bf 332} (2014), 1381--1383.

\bibitem{Iv} S. Ivanov, \emph{Heterotic supersymmetry, anomaly cancellation and equations of motion},
Phys. Lett. \textbf{B 685} (2010), 190--196.

\bibitem{FOU} A. Fino, A. Otal, L. Ugarte, \emph{Six dimensional solvmanifolds with
holomorphically trivial canonical bundle},
Int. Math. Res. Not. IMRN 2015, no. {\bf 24}, 13757--13799.

\bibitem{AG} E. Abbena, A. Grassi, \emph{Hermitian left invariant metric on complex Lie groups and cosymplectic Hermitian manifolds},
Boll. Un. Mat. Ital. A (6) {\bf 5} (1986), 371--379.

\bibitem{Mostow} G.D. Mostow, \emph{Factor spaces of solvable groups},
Ann. of Math. \textbf{60} (1954), 1--27.

\bibitem{AV} A. Andrada, R. Villacampa,
\emph{Abelian balanced Hermitian structures on unimodular Lie algebras},
Transform. Groups {\bf 21} (2016), 903--927.

\bibitem{Yachou} A. Yachou, \emph{Une classe de vari\'{e}t\'{e}s semi-k\"ahl\'{e}riennes},
C. R. Acad. Sci. Paris, Ser. I. {\bf 321} (1995), 763--765.

\end{thebibliography}
\end{document}